\documentclass[10pt,english]{amsart}
\usepackage[T1]{fontenc}
\usepackage[latin1]{inputenc}
\usepackage{amstext}
\usepackage{amsthm}
\usepackage{amssymb}

\makeatletter

\providecommand{\tabularnewline}{\\}

\numberwithin{equation}{section}
\numberwithin{figure}{section}
\theoremstyle{plain}
\newtheorem{thm}{\protect\theoremname}
  \theoremstyle{plain}
  \newtheorem{lem}[thm]{\protect\lemmaname}
  \theoremstyle{plain}
  \newtheorem{prop}[thm]{\protect\propositionname}
  \theoremstyle{remark}
  \newtheorem{rem}[thm]{\protect\remarkname}
  \theoremstyle{plain}
  \newtheorem{cor}[thm]{\protect\corollaryname}
  \theoremstyle{plain}
  \newtheorem{conjecture}[thm]{\protect\conjecturename}

\makeatother

\makeatother

\usepackage{babel}
  \providecommand{\conjecturename}{Conjecture}
  \providecommand{\corollaryname}{Corollary}
  \providecommand{\lemmaname}{Lemma}
  \providecommand{\propositionname}{Proposition}
  \providecommand{\remarkname}{Remark}
\providecommand{\theoremname}{Theorem}

\begin{document}

\title[Construction of Nikulin configurations]{Construction of Nikulin configurations on some Kummer surfaces and applications}

\addtolength{\textwidth}{0mm}
\addtolength{\hoffset}{-0mm} 
\addtolength{\textheight}{0mm}
\addtolength{\voffset}{-0mm} 

\global\long\global\long\def\Alb{{\rm Alb}}
 \global\long\global\long\def\Jac{{\rm Jac}}
\global\long\global\long\def\Disc{{\rm Disc}}

\global\long\global\long\def\Tr{{\rm Tr}}
 \global\long\global\long\def\NS{{\rm NS}}
\global\long\global\long\def\PicVar{{\rm PicVar}}
\global\long\global\long\def\Pic{{\rm Pic}}
\global\long\global\long\def\Br{{\rm Br}}
 \global\long\global\long\def\Pr{{\rm Pr}}
\global\long\global\long\def\Km{{\rm Km}}
\global\long\global\long\def\rk{{\rm rk}}

\global\long\global\long\def\Hom{{\rm Hom}}
 \global\long\global\long\def\End{{\rm End}}
 \global\long\global\long\def\aut{{\rm Aut}}
 \global\long\global\long\def\NS{{\rm NS}}
 \global\long\global\long\def\SSm{{\rm S}}
 \global\long\global\long\def\psl{{\rm PSL}}
 \global\long\global\long\def\CC{\mathbb{C}}
 \global\long\global\long\def\BB{\mathbb{B}}
 \global\long\global\long\def\PP{\mathbb{P}}
 \global\long\global\long\def\QQ{\mathbb{Q}}
 \global\long\global\long\def\RR{\mathbb{R}}
 \global\long\global\long\def\FF{\mathbb{F}}
 \global\long\global\long\def\DD{\mathbb{D}}
 \global\long\global\long\def\NN{\mathbb{N}}
 \global\long\global\long\def\ZZ{\mathbb{Z}}
 \global\long\global\long\def\HH{\mathbb{H}}
 \global\long\global\long\def\Gal{{\rm Gal}}
 \global\long\global\long\def\OO{\mathcal{O}}
 \global\long\global\long\def\pP{\mathfrak{p}}
 \global\long\global\long\def\pPP{\mathfrak{P}}
 \global\long\global\long\def\qQ{\mathfrak{q}}

\global\long\global\long\def\mm{\mathcal{M}}
 \global\long\global\long\def\aaa{\mathfrak{a}}
\global\long\def\a{\alpha}
\global\long\def\b{\beta}
 \global\long\def\d{\delta}
 \global\long\def\D{\Delta}
\global\long\def\L{\Lambda}
 \global\long\def\g{\gamma}
 \global\long\def\G{\Gamma}
 \global\long\def\d{\delta}
 \global\long\def\D{\Delta}
 \global\long\def\e{\varepsilon}
 \global\long\def\k{\kappa}
 \global\long\def\l{\lambda}
 \global\long\def\m{\mu}
 \global\long\def\o{\omega}
 \global\long\def\p{\pi}
 \global\long\def\P{\Pi}
 \global\long\def\s{\sigma}
 \global\long\def\S{\Sigma}
 \global\long\def\t{\theta}
 \global\long\def\T{\Theta}
 \global\long\def\f{\varphi}
 \global\long\def\deg{{\rm deg}}
 \global\long\def\det{{\rm det}}
 \global\long\def\dps{{\displaystyle }}
 \global\long\def\Dem{D\acute{e}monstration: }
 \global\long\def\ker{{\rm Ker\,}}
 \global\long\def\im{{\rm Im\,}}
 \global\long\def\rg{{\rm rg\,}}
 \global\long\def\car{{\rm car}}
\global\long\def\fix{{\rm Fix( }}
 \global\long\def\card{{\rm Card\  }}
 \global\long\def\codim{{\rm codim\,}}
 \global\long\def\coker{{\rm Coker\,}}
 \global\long\def\mod{{\rm mod }}
 \global\long\def\pgcd{{\rm pgcd}}
 \global\long\global\long\def\qa{\mathfrak{a}}

\global\long\def\ppcm{{\rm ppcm}}
 \global\long\def\la{\langle}
 \global\long\def\ra{\rangle}

\subjclass[2000]{Primary: 14J28 ; Secondary: 14J50, 14J29, 14J10}


\author{Xavier Roulleau, Alessandra Sarti}
\begin{abstract}
A Nikulin configuration is the data of $16$ disjoint smooth rational
curves on a K3 surface. According to a well known result of Nikulin,
if a K3 surface contains a Nikulin configuration $\mathcal{C}$, then
$X$ is a Kummer surface $X=\Km(B)$ where $B$ is an Abelian surface
determined by $\mathcal{C}$. Let $B$ be a generic Abelian surface
having a polarization $M$ with $M^{2}=k(k+1)$ (for $k>0$ an integer)
and let $X=\Km(B)$ be the associated Kummer surface. To the natural
Nikulin configuration $\mathcal{C}$ on $X=\Km(B)$, we associate
another Nikulin configuration $\mathcal{C}'$; we denote by $B'$
the Abelian surface associated to $\mathcal{C}'$, so that we have
also $X=\Km(B')$. For $k\geq2$ we prove that $B$ and $B'$ are
not isomorphic. We then construct an infinite order automorphism of
the Kummer surface $X$ that occurs naturally from our situation.
Associated to the two Nikulin configurations $\mathcal{C},$ $\mathcal{C}'$,
there exists a natural bi-double cover $S\to X$, which is a surface
of general type. We study this surface which is a Lagrangian surface
in the sense of Bogomolov-Tschinkel, and for $k=2$ is a Schoen surface.
\end{abstract}

\maketitle

\section{Introduction}

To a set $\mathcal{C}$ of $16$ disjoint smooth rational curves $A_{1},\dots,A_{16}$
on a K3 surface $X$, Nikulin proved that one can associate a double
cover $\tilde{B}\to X$ branched over the curve $\sum A_{i}$, such
that the minimal model $B$ of $\tilde{B}$ is an Abelian surface
and the $16$ exceptional divisors of $\tilde{B}\to B$ are the curves
above $A_{1},\dots,A_{16}$. The K3 surface $X$ is thus a Kummer
surface.

We call a set of $16$ disjoint $(-2)$-curves on a K3 surface a \textit{Nikulin
configuration}. Let us recall a classical construction of Nikulin
configurations. The Kummer surface $X=\Km(B)$ of a Jacobian surface
$B$ can be embedded birationally onto a quartic $Y$ of $\PP^{3}$
with $16$ nodes. Projecting from one node  one gets
another projective model for $X$, this is a double cover $Y'\to \mathbb P ^2$ of the plane
%gives a double cover $\pi:X\to\PP^{2}$ of the plane, 
branched over $6$ lines tangent to a conic. The strict
transform (in $X$) of that conic is the union of two $(-2)$-curves
$A_{1},A_{1}'$, with $A_{1}A_{1}'=6$. One of these two curves, $A_{1}$
say, corresponds to the node from which we project. Above the $15$
intersection points of the $6$ lines there are $15$ disjoint $(-2)$-curves
$A_{2},\dots,A_{16}$ on $X$, which corresponds to the $15$ other
nodes of the quartic $Y$.\\
The divisors $\mathcal{C}=\sum_{i=1}^{16}A_{i},\,\mathcal{C}'=A_{1}'+\sum_{i=2}^{16}A_{i}$
are two Nikulin configurations. The Abelian surface $B$ is then the
Jacobian of the double cover of $A_{1}$ branched over $A_{1}\cap A_{1}'$. 

Let now $k>0$ be an integer and let $(B,M)$ be a polarized Abelian
surface with $M^{2}=k(k+1)$, such that $B$ is generic, i.e. $\NS(B)=\ZZ M$.
Let $X=\Km(B)$ be the associated Kummer surface, let $L\in\NS(X)$
be the class corresponding to $M$ (so that $L^{2}=2M^{2}$), and
let $\mathcal{C}=A_{1}+\dots+A_{16}$ be the natural Nikulin configuration
on $\Km(B)$ (the class $L$ is orthogonal to the $A_{i}$'s). We
obtain the following results, which for $k=1$ are the results we
recalled for Jacobian Kummer surfaces:
\begin{thm}
\label{thm:main1}Let be $t\in\{1,\dots,16\}$. There exists a $(-2)$-curve
$A_{t}'$ on $\Km(B)$ such that $A_{t}A_{t}'=4k+2$ and $\mathcal{C}_{t}=A_{t}'+\sum_{j\neq t}A_{j}$
is another Nikulin configuration. \\
The numerical class of $A_{t}'$ is $2L-(2k+1)A_{t}$; the class
\[
L_{t}'=(2k+1)L-2k(k+1)A_{t}
\]
generates the orthogonal complement of the $16$ curves $A'_{t}$
and $\{A_{j}\,|\,j\neq t\}$; moreover $L_{t}'^{2}=L^{2}$.
\end{thm}
A \textit{Kummer structure} on a Kummer surface $X$ is an isomorphism class
of Abelian surfaces $B$ such that $X\simeq\Km(B)$. It is known that
Kummer structures on $X$ are in one-to-one correspondence with the
orbits of Nikulin configurations by the action of the automorphism
group of $X$ (see Proposition \ref{prop:The-Kummer-structures}).
In \cite[Question 5]{Sh1}, Shioda raised the question whether if there could be more than one
Kummer structure on a Kummer surface. In \cite{GH}, Gritsenko and
Hulek noticed that $\Km(B)\simeq\Km(B^{*})$, where $B^{*}$ is the
dual of $B$, a $(1,t)$-polarized Abelian surface (thus $B\not\simeq B^{*}$
if $t>1$). In \cite{HLOY} Hosono, Lian, Oguiso and Yau proved that
the number of Kummer structures is always finite and they construct
for any $N\in\NN^{*}$ a Kummer surface of Picard number $18$ with
at least $N$ Kummer structures. When the Picard number is $17$ (which
is the case of our paper), by results of Orlov \cite{Orlov} on derived
categories, the number of Kummer structures on $X$ equals $2^s$ where $s$ is the number
of prime divisors of $\frac{1}{2}M^{2}$. In Section \ref{subsec:Nikulin-structures-and Fermat},
we obtain the following result
\begin{thm}
\label{thm:Main 3}Suppose $k\geq2$. There is no automorphism of
$X$ sending the Nikulin configuration $\mathcal{C}=\sum_{j=1}^{16}A_{j}$
to the configuration $\mathcal{C}_{t}=A_{t}'+\sum_{j\neq t}A_{j}$. 
\end{thm}
Therefore the two configurations $\mathcal{C},\,\mathcal{C}_{t}$
belong in two distinct orbits of Nikulin configurations under the
action of $\aut(X)$. As far as we know, Theorem \ref{thm:Main 3}
gives the first explicit construction of two distinct Kummer structures
on a Kummer surface: the constructions in \cite{HLOY} and \cite{GH}
use lattice theory and do not give a geometric description of the
Nikulin configurations.

We already recalled that when $X$ is a Jacobian Kummer surface, there
exists a non-symplectic involution $\iota$ on $X$ such that the
double cover $\pi:X\to\PP^{2}$ is the quotient of $X$ by $\iota$
(after contraction of the $16$ $(-2)$-curves). That involution
exchanges the $(-2)$-curves $A_{1}$ and $A_{1}'$ and fixes the
$15$ other curves $\{A_{j}\,|\,j\neq1\}$. For $X$ a K3 surface
with a polarization $L$ such that $L^{2}=2k(k+1)$ and $t\in\{1,\dots,16\}$,
let $\theta_{t}$ be the involution of $\NS(X)\otimes\QQ$ defined
by $L\to L_{t}'$, $A_{t}\to A_{t}'$ (as defined in Theorem \ref{thm:main1}),
and $\t_{t}(A_{j})=A_{j}$ for $j\neq t$. When $k=1$, $\theta_{1}$
is in fact the action of the involution $\iota$ on $\NS(X)$ : $\iota^{*}=\t_{1}$.
We do not have such an interpretation when $k>1$ (this is in fact
the content of Theorem \ref{thm:Main 3}), but we obtain the following
result on the product $\t_{i}\t_{j}$:
\begin{thm}
\label{thm:Main-2}For $1\leq i\neq j\leq16$ there exists an infinite
order automorphism $\mu_{ij}$ of $X$ such that the action of $\mu_{ij}$
on $\NS(X)$ is $\mu_{ij}^{*}=\t_{i}\t_{j}$ . 
\end{thm}
The classification of the automorphism group of a generic Jacobian
Kummer surface has been has been completed by Keum \cite{Keum} (who constructed
the last unknown automorphisms) and by Kondo \cite{Kondo} (who proved
that there was indeed no more automorphisms). 
We are far from such a knowledge for non Jacobian Kummer surfaces, thus it is interesting
to have a construction of such automorphisms $\mu_{ij}$. 
Let $A$ be an Abelian variety. In \cite{NarNori}, Narasimhan and Nori prove that the orbits by $\aut(A)$ of the principal polarisations in the N\'eron-Severi group $\NS(A)$ are finite. Similarly, one could think to prove that the number of Kummer structures on a K3 is finite by associating to each Nikulin configuration  $\mathcal{C}$ the pseudo-ample divisor $L_{\mathcal{C}}$ orthogonal to $\mathcal{C}$ and by proving that the number of orbits  of such $L_{\mathcal{C}}$ under the action of $\aut(X)$  in $\NS(X)$ is finite. 
Our approach is closer to that idea than to the  solutions previous 
proposed e.g. in \cite{HLOY} or \cite{GH}, and it gives us more informations on $\aut(X)$.\\
% That gives the following idea for another approach for proving that the number of Kummer structures on a K3 is finite. Indeed to a Nikulin configuration $\mathcal{C}$, one can attach the pseudo-ample divisor $L_{\mathcal{C}}$ orthogonal to $\mathcal{C}$. If one can prove that the number of orbits under the action of $\aut(X)$ of such $L_{\mathcal{C}}$  in $\NS(X)$ is finite, one obtained the finiteness of the number of Kummer structures.  It is also worth mentioning the connection to the paper of Narasimhan and Nori \cite{NarNori} which describes the action of  the automorphism group of an Abelian surface $X$ on the N\'eron-Severi group $\NS(X)$.\\
Observe that one can repeat the construction in Theorem \ref{thm:main1},
starting with configuration $\mathcal{C}_{i}$ instead of $\mathcal{C}$,
but Theorem \ref{thm:Main-2} tells us that the Nikulin configurations
so obtained will be in the orbit of the Nikulin configuration $\mathcal{C}$
under the automorphism group $X$, thus we do not obtain new Nikulin
structures in that way (observe also that $\mathcal{C}_{t}$ and $\mathcal{C}_{t'}$
($t\neq t'$) are in the same orbit).

The paper is organized as follows: In Section \ref{sec1:Two-Nikulin-configurations}
we construct the curve $A'_{i}$ such that $A_{i}A_{i}'=4k+2$ and
we prove Theorem \ref{thm:main1}. This is done by geometric considerations
on the properties of the divisor $L_{i}'$, which we prove is big
and nef.

In Section \ref{sec:Nikulin-configurations-and}, we construct the
automorphisms mentioned in Theorem \ref{thm:Main-2}. This is done
by using the Torelli Theorem for K3 surfaces. We then prove Theorem
\ref{thm:Main 3}, which is obtained by considerations on the lattice
$H^{2}(X,\ZZ)$.

In Section \ref{sec:bi-double-covers-associated}, we study the bi-double
cover $Z\to X$ associated to the two Nikulin configurations $\mathcal{C}=\sum_{i=1}^{16}A_{i},\,\mathcal{C}'=A_{1}'+\sum_{i=2}^{16}A_{i}$.
When $k=2$, $Y$ is a so-called Schoen surface, a fact that has been
already observed in \cite{RRS}. Schoen surfaces carry many remarkable
properties (see e.g. \cite{CMR, RRS}). For example the kernel of
the natural map
\[
\wedge^{2}H^{0}(Z,\Omega_{Z})\to H^{0}(Z,K_{Z})
\]
is one dimensional, and is not of the form $w_{1}\wedge w_{2}$, i.e.
by the Castelnuovo De Franchis Theorem, it does not come from a fibration
of $Z$ onto a curve of genus $\geq2$. Surfaces with this property
are called Lagrangian. We will see that for the other $k>1$, the surfaces 
are also Lagrangian. \\
In Subsection \ref{subsec:An-hyperelliptic-curve}, we discuss 
the singularities of the curve $A_{i}+A_{i}'$. The transversality of the intersection of
two rational curves on a K3 surface is an interesting
but open problem in general (see e.g. \cite{Huyb}). We also
study the curve $\Gamma_{i}$ on the Abelian surface $B$ coming from
the pull-back of the curve $A'_{i}$. That curve $\Gamma_{i}$ is
hyperelliptic and has a unique singularity, which is a point of multiplicity
$4k+2$, and therefore $\Gamma_{i}$ has geometric genus $\leq2g$.
In the case of a Jacobian surface, $\Gamma_{i}$ has been used as
the branch locus of covers of $B$ by Penegini \cite{Pene} and Polizzi
\cite{Polizzi}, for creating new surfaces of general type. We end
this paper by remarking that $\G_{i}$ is a curve with the lowest
known H-constant (see \cite{RoulleauIMRN} for definitions and motivations)
on an Abelian surface.

{\bf Acknowledgements} The authors thank the anonymous referee for useful remarks improving the exposition of the paper.

\section{Two Nikulin configurations on Kummer surfaces\label{sec1:Two-Nikulin-configurations}}

\subsection{Two rational curves $A_{1},\,A_{1}'$ such that $A_{1}A_{1}'=2(2k+1)$}

Let $k>0$ be an integer and let $B$ be an abelian surface with a
polarization $M$ such that $M^{2}=k(k+1)$. We suppose that $B$
is generic so that $M$ generates the Néron-Severi group of $B$.
Let $X=\Km(B)$ be the associated Kummer surface and $A_{1},\dots,A_{16}$
be its $16$ disjoint $(-2)$-curves coming from the desingularization
of $B/[-1]$. \\
By \cite[Proposition 3.2]{Morrison}, \cite[Proposition 2.6]{GS},
corresponding to the polarization $M$ on $B$, there is a polarization
$L$ on $\Km(B)$ such that 
\[
L^{2}=2k(k+1)
\]
and $LA_{i}=0,\,i\in\{1,\dots,16\}$. The Néron-Severi group of $X=\Km(B)$
satisfies: 
\[
\ZZ L\oplus K\subset\NS(X),
\]
where $K$ denotes the Kummer lattice (the saturated sub-lattice of $NS(X)$ containing the
$16$ classes $A_{i}$). For $B$ generic among polarized
Abelian surfaces $\rk(\NS(X))=17$ and $\NS(X)$ is an overlattice
of finite index of $\ZZ L\oplus K$ which is described precisely in
\cite{GS}, in particular we will use the following result:
\begin{lem}
\label{lem:At most 4 beta}(\cite[Remarks 2.3 \& 2.10]{GS}) An element
$\Gamma\in\NS(X)$ has the form $\Gamma=\alpha L-\sum\beta_{i}A_{i}$
with $\alpha,\beta_{i}\in\frac{1}{2}\ZZ$. If $\alpha$ or $\beta_{i}$
for some $i$ is in $\frac{1}{2}\ZZ\setminus\ZZ,$ then at least $4$
of the $\beta_{j}$'s are in $\frac{1}{2}\ZZ\setminus\ZZ$, if moreover
$\alpha\in\ZZ$, at least $8$ of the $\beta_{j}$'s are in $\frac{1}{2}\ZZ\setminus\ZZ$.
\end{lem}
The divisor 
\[
A_{1}'=2L-(2k+1)A_{1}
\]
is a $(-2)$-class, indeed:
\[
(2L-(2k+1)A_{1})^{2}=8k(k+1)-2(2k+1)^{2}=-2,
\]
and one has $A_{1}'A_{i}=0$ for $i=2,\cdots,16$. By the Riemann-Roch Theorem and since $LA_1 '>0$, the class $A_1 '$ is represented by an effective divisor. Let us prove the
following result
\begin{thm}
\label{thm:The-class-is -2}The class $A_{1}'$ can be represented
by a $(-2)$-curve and $A_{1}A'_{1}=2(2k+1)$. The set of $(-2)$-curves
\[
A_{1}',A_{2},\dots, A_{16}
\]
 is another Nikulin configuration on $X$.
\end{thm}
In order to prove Theorem \ref{thm:The-class-is -2}, let us define
\[
L'=(2k+1)L-2k(k+1)A_{1}.
\]
One has $L'A_{1}'=0$ and 
\[
L'^{2}=(2k+1)^{2}2k(k+1)-8k^{2}(k+1)^{2}=2k(k+1)=L^{2}.
\]
First let us prove:
\begin{prop}
\label{prop:Suppose-.-Thena)}One has:\\
a) The divisor $L'$ is nef and big. Moreover a $(-2)$-class $\Gamma$ satisfies $\Gamma L'=0$ if and only if $\Gamma=A_1 '$ or $\Gamma=A_j$ for $j$ in $\{2,..., 16\}$.\\
b) The linear system $|L'|$ has no base components.\\
c) The linear system $|L'|$ defines a morphism from $X=\Km(B)$ to
$\PP^{k^{2}+k+1}$ which is birational onto its image and contracts
the divisor $A_{1}'$ and the $15$ $(-2)$-curves $A_{i},\,i\geq2$.

\end{prop}
\begin{proof}
\textbf{Proof of a).} We already know that $\ensuremath{L'^{2}=2k(k+1)>0}$.
By the Riemann-Roch Theorem either $L'$ or $-L'$ is effective. Since $LL'>0$,
we see that $L'$ is effective. On a K3 surface, the $(-2)$-curves are the only
irreducible  curves with negative self-intersection, thus $L'$ is nef if and only
if $L'\Gamma\geq0$ for each irreducible $(-2)$-curve $\Gamma$.
Let 
\[
\Gamma=\alpha L-\sum_{i=1}^{16}\beta_{i}A_{i},\qquad\alpha,\beta_{i}\in\frac{1}{2}\ZZ
\]
be the class of $\Gamma$ in $\NS(X)$. Since $\Gamma$ represents
an irreducible curve we have $\alpha\geq0$. Moreover if $\Gamma=A_{i}$
then the condition $L'\Gamma\geq0$ is trivially verified so that we can assume $\Gamma A_{i}\geq0$,
which gives $\beta_{i}\geq0$. From the condition $\Gamma^{2}=-2$,
we get 
\begin{equation}
k(k+1)\alpha^{2}-\sum_{i}\beta_{i}^{2}=-1\label{eq:carre}
\end{equation}
Assume that the $(-2)$-curve $\Gamma$ satisfies $L'\Gamma<0$. We
have 
\[
0>L'\Gamma=\left((2k+1)L-2k(k+1)A_{1}\right)\Gamma=2\alpha k(k+1)(2k+1)-4k(k+1)\beta_{1},
\]
thus 
\[
\beta_{1}>\frac{(2k+1)}{2}\alpha. \label{eq:BETA}
\]
Combining with equation \eqref{eq:carre} we get 
\[
-1=k(k+1)\alpha^{2}-\sum_{i}\beta_{i}^{2}<-\frac{1}{4}\alpha^{2}-\sum_{i=2}^{15}\beta_{i}^{2}.
\]
which is 
\begin{equation}
\frac{1}{4}\alpha^{2}+\sum_{i=2}^{15}\beta_{i}^{2}<1\label{eq:somme limite}
\end{equation}
thus $\alpha\in\{0,1/2,1,3/2\}$. \\
If $\alpha=0$, by \eqref{eq:carre} either exactly one of the $\beta_{i}=1$ (but             
this is not possible since it would give $\Gamma=-A_{i}$) or exactly
$4$ of the $\beta_{i}'s$ are equal to $\frac{1}{2}$ and the others
are $0$ but such a class is not contained in $\NS(X)$ by Lemma \ref{lem:At most 4 beta}.\\
If $\alpha=\frac{1}{2}$, then  from inequality \eqref{eq:somme limite},
$\beta_{i}\in\{0,\frac{1}{2}\}$ for $i\geq2$ and at most $3$ of
these $\beta_{i}$'s equal $\frac{1}{2}$.  By Lemma \ref{lem:At most 4 beta}
at least $4$ of the $\beta_{i}$ are in $\frac{1}{2}\ZZ\setminus\ZZ$,
thus $3$ of the $\beta_{i},\,i\geq2$ equals $\frac{1}{2}$ and the
others are $0$. Then from equation \eqref{eq:carre}, we get:
\[
\beta_{1}^{2}=\frac{k^{2}+k+1}{4}.
\]
Suppose that there exists $n\in \mathbb{N}$ such that $k^{2}+k+1=n^2$. 
Then $n> k$, but since $n^2\geq (k+1)^2>k^{2}+k+1$, we get a contradiction. 
Hence $\forall k\in\NN^{*}$, the integer $k^{2}+k+1$
is never a square and therefore the case $\alpha=\frac{1}{2}$ is impossible. \\
If $\alpha=1$, at most $2$ of the $\beta_{i}$'s with $i>1$ are
equal $\frac{1}{2}$ and the others are $0$, by applying Lemma \ref{lem:At most 4 beta}
we get $\beta_{i}=0$ for $i>1$ and $\beta_{1}\in\NN$. Then equation
\eqref{eq:carre} implies 
\[
\beta_{1}^{2}=k^{2}+k+1,
\]
which we know has no integral solutions for $k>0$.\\
If $\alpha=\frac{3}{2}$, at most $1$ of the $\beta_{i}$'s with
$i>1$ is $\frac{1}{2}$, this is also impossible by Lemma \ref{lem:At most 4 beta},
therefore such $\Gamma$ does not exist and this concludes the proof
that $L'$ is big and nef for all $k\geq1$.\\
Assume that the $(-2)$-curve $\Gamma$ satisfies $L'\Gamma=0$ and is not $A_j$ for $j\geq 2$. Then one has $\beta_{1}=\frac{(2k+1)}{2}\alpha $, and one computes that either $\a =2$, $\b_1=2k+1$ and $\Gamma=A_1 '$, or $\a=1, \b_1=\frac{(2k+1)}{2}\alpha$ and (up to re-ordering) $\b_2=b_3=b_4=1/2$. Since $\a$ is an integer the second case is impossible by Lemma \ref{lem:At most 4 beta}.

\textbf{Proof of b).} By \cite[Section 3.8]{reid} either $|L'|$
has no fixed part or $L'=aE+\Gamma$, where $|E|$ is a free pencil,
and $\Gamma$ a $(-2)$-curve with $E\Gamma=1$. In that case, write
$\Gamma=\alpha L-\sum\beta_{i}A_{i}$. Then 
\[
2k(k+1)=L'^{2}=2a-2
\]
gives $a=k^{2}+k+1$. In particular, $a$ is odd. But 
\[
a-2=L'\Gamma=2k(k+1)(2k+1)\alpha-4k(k+1)\beta_{1}
\]
and since $\alpha,\beta_{1}\in\frac{1}{2}\ZZ,$ one gets that $a$
is even, which yields a contradiction. Therefore $|L'|$ has no base
components. By \cite[Corollary 3.2]{SD}, it then has no base
points.

\textbf{Proof of c).} The linear system $|L'|$ is big and
nef without base points. We have to show that the resulting morphism has degree one, i.e.
that $|L'|$ is not hyperelliptic (see \cite[Section 4]{SD}). By
loc. cit., $|L'|$ is hyperelliptic if there exists a genus $2$ curve
$C$ such that $L'=2C$ or there exists an elliptic curve $E$ such
that $L'E=2$.\\
 In the first case $L'^{2}=8$, but since $L'^{2}=2k(k+1)$, that
cannot happen. Assume now 
\[
E=\alpha L-\sum\beta_{i}A_{i},
\]
for $E$ with $EL'=2$, we get 
\[
2=\left(\alpha L-\sum\beta_{i}A_{i}\right) \left((2k+1)L-2k(k+1)A_{1}\right)=k(k+1)\left(2(2k+1)\alpha-2\beta_{1}\right).
\]
Since $\alpha,\beta_{1}\in\frac{1}{2}\ZZ$, $2(2k+1)\alpha-2\beta_{1}$
is an integer, thus we get $k=1$ and $6\alpha-2\beta_{1}=1$. Since
$E^{2}=0$, one obtain 
\[
2\alpha^{2}=\sum\beta_{i}^{2},
\]
using $\beta_{1}=3\alpha-\frac{1}{2}$, one reaches a contradiction.
\\
Therefore $|L'|$ defines a birational map $X\to\PP^{N}$ onto its
image, contracting the $(-2)$-curves $\Gamma$ such that $L'\Gamma=0$,
moreover $N=h^{0}(L')-1=\frac{L'^{2}}{2}+1=k^{2}+k+1$.
\end{proof}
We can now prove Theorem \ref{thm:The-class-is -2}:
\begin{proof}
We proved that the only $(-2)$-classes that are contracted by $L'$ are $A_1'$, $ A_2,$ $\dots,A_{16}$. We know moreover that $A_1'A_j=A_i A_j=0$ for $2 \leq i \neq j\leq 16$. 
Since one has $L'A_{1}'=0$ the base point free linear system $|L'|$ contracts the connected components of $A_{1}'$ to some points. 
Therefore by the Grauert contraction Theorem (see \cite[Chapter III, Theorem 2.1]{BPVdV}), 
the support of $A_1'$ is the union of irreducible curves $(C_i)_{i\in \{1,\dots,m\}}$ (for $m\in \mathbb{N},\,m\neq 0$)
 such that the intersection matrix $(C_i C_j)$ is negative definite. \\
Since $X$ is a K3 surface, the curves $C_i$ are $(-2)$-curves. Since $L'$ only contracts the (-2)-classes 
$A_1'$, $ A_2,$ $\dots,A_{16}$ that are disjoint, we get that $m=1$ and  we conclude that $A_{1}'$ is the class of a $(-2)$-curve $C_1$.

%These curves $C_i$ are disjoint from 
%the curves $A_2,\dots,A_{16}$, and $L'C_i=0$. Thus the Picard number of $X$ is at least $16+m$.
%Since the Picard number
%of $X=\Km(B)$ is $17$ (we recall that we took $B$ generic), we see that $m=1$ 
%and we conclude that $A_{1}'$ is the class of a $(-2)$-curve $C_1$.
\end{proof}

\subsection{A projective model of the surface $\protect\Km(B)$}

Let us describe a natural map from $\Km(B)$ to $\PP^{k+1}$, which
is birational for $k>1$:
\begin{thm} \label{thm7}
The class $D=L-kA_{1}$ is big and nef with 
\[
(L-kA_{1})^{2}=2k
\]
and for $k\geq2$ it defines a birational map 
\[
\phi:\Km(B)\to\PP^{k+1}
\]
onto its image  $X$ such that $X$ (of degree $2k$) has $15$ ordinary double
points and moreover the curves $A_{1}'$ and $A_{1}$ are sent to
two rational curves of degree $2k$ such that $A_{1}A_{1}'=2(2k+1)$.
\end{thm}
\begin{rem} 
We have 
\[
A_{1}'+A_{1}=2(L-kA_{1})
\]
so that $A_{1}'+A_{1}$ is cut out by a quadric of $\PP^{k+1}$ and
is $2$-divisible. \\
%B) As pointed out to us by a referee, one can use the proof of Theorem \ref {thm7} combined with some arguments in the proof of Theorem \ref{thm:The-class-is -2}  to prove again that $A_1'$ is the class of a $(-2)$-curve.
\end{rem}
\begin{proof}
We proceed as in Proposition \ref{prop:Suppose-.-Thena)}.\\
\textbf{Let us show that $D$ is nef and big.} We have to prove that
$D\Gamma\geq0$ for each irreducible $(-2)$-curve $\Gamma$. As above,
let 
\[
\Gamma=\alpha L-\sum\beta_{i}A_{i},\qquad\alpha,\beta_{i}\in\frac{1}{2}\ZZ,
\]
be such that $\Gamma D<0$. Then 
\[
\Gamma D=2\alpha k(k+1)-2k\beta_{1}<0,
\]
implies $\beta_{1}>(k+1)\alpha.$\\
Combining with the equation \eqref{eq:carre}, we get
\[
1>(k+1)\alpha^{2}+\sum_{i\geq2}\beta_{i}^{2},
\]
thus $\alpha<1$. As in Proposition \ref{prop:Suppose-.-Thena)},
the case $\alpha=0$ is impossible. If $\alpha=\frac{1}{2}$, then
$k\in\{1,2\}$, but as above, Lemma \ref{lem:At most 4 beta} implies
that this is not possible. Thus $D$ is nef and big.\\
Let us now suppose $k>1$. \textbf{Let us show that $|D|$ has no
base components}. Suppose that there is a base component. Then $D=aE+\Gamma$,
where $a\in\NN$, $|E|$ is a free pencil, $\Gamma$ is a $(-2)$-curve
and $E\Gamma=1$. One has 
\[
2k=D^{2}=2a-2,
\]
thus $a=k+1$, so that
\[
L-kA_{1}=(k+1)E+\Gamma.
\]
Suppose that $\Gamma=A_{1}$, then $2k=A_{1}D=k-1$ and $k=-1$, which
is impossible. If $\Gamma=A_{i},$ $i\geq2$, then $0=DA_{i}=k-1,$
thus $k=1$, but we assumed that $k>1$.\\
Thus we can assume that $\Gamma$ is not one of the $A_{i}$ and write
$\Gamma=\alpha L-\sum\beta_{i}A_{i}$ with $\alpha,\beta_{i}\geq0$.
One has 
\begin{equation}
2k=DA_{1}=(k+1)EA_{1}+2\beta_{1},\label{eq:2k}
\end{equation}
moreover 
\begin{equation}
2k(k+1)=(L-kA_{1})L=(k+1)EL+2k(k+1)\alpha.\label{eq:2k(k+1)}
\end{equation}
Since $EA_{1}\geq0$ we obtain from equation \eqref{eq:2k} that either
$\beta_{1}=k$ (and $EA_{1}=0$) or $\beta_{1}=\frac{k-1}{2}$ and
$EA_{1}=1$, in that second case since 
\[
E(L-kA_{1})=E((k+1)E+\Gamma)=1
\]
one obtains $EL=k+1.$\\
Since $EL\geq0$, we obtain from equation \eqref{eq:2k(k+1)} that
$\alpha\in\{0,\frac{1}{2},1\}$, but as in Proposition \ref{prop:Suppose-.-Thena)},
$\alpha=0$ is not possible. Moreover if $\alpha=1$, $EL=0$, but
this contradicts the Hodge Index Theorem since $E^{2}=0$ and $L^{2}>0$,
therefore $\alpha=\frac{\ensuremath{1}}{2}$. If $\beta_{1}=k$, from
$\Gamma^{2}=-2$, one gets 
\[
\frac{k(k+1)}{4}-k^{2}-\sum_{i\geq2}\beta_{i}^{2}=-1
\]
which is 
\[
\sum_{i\geq2}\beta_{i}^{2}=\frac{1}{4}(-3k^{2}+k+4).
\]
 But for $k>1$, $-3k^{2}+k+4<0$ and we obtain a contradiction. If
now $\beta_{1}=\frac{k-1}{2}$, then $EL=k+1$, but equation \eqref{eq:2k(k+1)}
gives $EL=k$, contradiction. Therefore $|D|$ has no base component.\\
\textbf{Let us show that $|D|$ defines a birational map.} We have
to show that $|D|$ is not hyperelliptic. Suppose that $D=2C$ where
$C$ is a genus $2$ curve. Then $D^{2}=8$; since $D^{2}=2k$, we
get $k=4$. One has $D=L-4A_{1}$ and the class of $C$ is $\frac{1}{2}L-2A_{1}$.
Then $\frac{1}{2}L\in\NS(X)$, which contradicts the fact that $L$
 generates the orthogonal complement of $\NS(\Km(B)),$ and so $L$ is primitive.
%generates the orthogonal complement of $\NS(\Km(B))$ and so $L$ is primitive.\\
Suppose now that there exists an elliptic curve $E$ such that $DE=2$.
Let 
\[
E=\alpha L-\sum\beta_{i}A_{i},
\]
with $\alpha\in\frac{1}{2}\ZZ$. Since $D=L-kA_{1},$ one has 
\[
DE=2k(k+1)\alpha-2k\beta_{1},
\]
therefore $k(k+1)\alpha-k\beta_{1}=1$. If $\alpha\in\ZZ,$ then if
$\beta_{1}\in\ZZ$, one gets $k=1$, if $\beta_{1}=\frac{b}{2}$ with
$b$ odd, then 
\[
k(2(k+1)\alpha-b)=2
\]
and $k=2$ (we supposed $k>1$), $6\alpha-b=2,$ which is impossible
since $b$ is odd. If $\alpha=\frac{a}{2}$ with $a\in\ZZ$ odd ,
then $k((k+1)a-2\beta_{1})=2$. Then since $2\beta_{1}\in\ZZ$ and
$k>1$, one has $k=2$ and $3a-2\beta_{1}=1$, thus $\beta_{1}=\frac{3a-1}{2}=3\alpha-\frac{1}{2}\in\ZZ$.
We have moreover (since $k=2$):
\[
0=E^{2}=6\alpha^{2}-\sum\beta_{i}^{2}
\]
thus 
\[
9\alpha^{2}-3\alpha+\frac{1}{4}+\sum_{i\geq2}\beta_{i}^{2}=6\alpha^{2},
\]
and $3\alpha^{2}-3\alpha+\frac{1}{4}\leq0$, the only possibility
is $\alpha=\frac{1}{2}$, but then $\sum_{i\geq2}\beta_{i}^{2}=\frac{1}{2}$,
which is impossible since, by Lemma \ref{lem:At most 4 beta}, there is no class with $\beta_{i}=\frac{1}{2}$
for only $2$ indices $i$. Therefore when $k>1$, $|D|$ defines
a birational map to $\PP^{N}$, with $N=\frac{D^{2}}{2}+1=k+1$. That
maps contracts the curves $\Gamma$ with $\Gamma D=0$, ie $A_{2},\dots,A_{16}$.

One has
\[
A_{1}(L-kA_{1})=2k=A_{1}'(L-kA_{1}),
\]
thus the curves $A_{1},A_{1}'$ in $\PP^{k+1}$ have degree $2k$.
Moreover $A_{1}A_{1}'=2(2k+1)$.\\
Let us prove that the 15 $(-2)$-curves $A_{i},\,i>1$ are the only
ones contracted i.e. they are the only solutions of the equation $\Gamma D=0$,
($D=L-kA_{1}$). Suppose $\Gamma\neq A_{i}$, $\Gamma=\alpha L-\sum\beta_{i}A_{i}$.
One has $\Gamma D=0$ if and only if
\[
\alpha(k+1)=\beta_{1},
\]
and $\alpha^{2}k(k+1)-\sum\beta_{i}^{2}=-1$, which gives 
\[
(k+1)\alpha^{2}+\sum_{i>1}\beta_{i}^{2}=1,
\]
which has no solutions by Lemma \ref{lem:At most 4 beta}. 
\end{proof}
\begin{rem}
\label{rem:To-the-pair}To the pair $(L,A_{1})$ one can associate
the pair $(L',A_{1}')$, with 
\[
L'=(2k+1)L-2k(k+1)A_{1},\,\,A_{1}'=2L-(2k+1)A_{1}
\]
with the same numerical properties 
\[
L^{2}=L'^{2}=2k,\,LA_{1}=0=L'A_{1}',\,LA_{1}'=4k(k+1)=L'A_{1}.
\]
The polarization $L'$ comes from a polarization $M'$ on the Abelian
surface $B'$ associated to the Nikulin configuration $A_{1}',A_{2},\dots,A_{16}$.
We will see that for $k=1$ the mapping $\Psi:(L,A_{1})\to(L',A_{1}')$
is an involution of $\NS(X)$ which comes from an involution of $X$,
and the Abelian surfaces $B$, $B'$ are isomorphic. \\
One can repeat the construction with $(L',A_{2})$ instead of $L,A_{1}$
etc... Let us define the maps $\Psi_{i},\,\Psi_{j}$, $\{i,j\}=\{1,2\}$
by $\Psi_{i}(L)=(2k+1)L-2k(k+1)A_{i}$, $\Psi_{i}(A_{i})=2L-(2k+1)A_{i}$,
$\Psi_{i}(A_{j})=A_{j}$. It is easy to check that $\Psi_{1}\circ\Psi_{2}$
has infinite order, and we therefore obtain in that way an infinite
number of Nikulin configurations. For any $k\in \mathbb{N},\,k\neq0$, we will see that the
map $\Psi_{i}\circ\Psi_{j}$ for $i\neq j$ is in fact the restriction
of the action of an automorphism of $X$ on $\NS(X)$. 
\end{rem}

\subsection{The first cases $k=1,2,3,4$}

In this subsection, we give a more detailed description of our construction
when $k$ is small. One has \\
\begin{tabular}{|c|c|c|c|c|}
\hline 
$k$ & $1$ & $2$ & $3$ & $4$\tabularnewline
\hline 
$A_{1}A_{1}'$ & $6$ & $10$ & $14$ & $18$\tabularnewline
\hline 
$L^{2}$ & $4$ & $12$ & $24$ & $40$\tabularnewline
\hline 
\end{tabular}\\
and the morphism $\phi$ associated to the linear system $|L-kA_{1}|$
is from $\Km(B)$ to $\PP^{k+1}$, with $k+1=2,3,4,5$ (which produce
the most famous geometric examples of K3 surfaces).

The case $k=1$ has been discussed in the Introduction. 

For $k=2$, the result was already observed in \cite{RRS}. The image
of $\phi$ is a $15$-nodal quartic $Q=Q_{4}$ in $\PP^{3}$, the
curves $A_{1},A_{1}'$ are sent to two degree $4$ rational curves
(denoted by the same letters) meeting in $10$ points. As we already
observed, the divisor $A_{1}+A_{1}'$ is a {\it 2-divisible class}.
The double cover $Y\to Q$ branched over $A_{1}+A_{1}'$ has $40$
ordinary double points coming from the $15$ singular points on $Q$
and from the $10$ intersection points of $A_{1}$ and $A_{1}'$.
This surface $Y$ is described in \cite{RRS}. It is a general type
surface, a complete intersection in $\PP^{4}$ of a quadric and the
Igusa quartic. It is the canonical image of its minimal resolution.
The double cover $S$ of $Y$ branched over the $40$ nodes is a so-called
Schoen surface. It is a surface with $p_{g}(S)=p_{g}(Y)=5$, thus
the canonical image of $S$ is $Y$ and the degree of the canonical
map of the Schoen surface is $2$.

For $k=3$, one get a model $Q_{6}$ of $X$ in $\PP^{4}$ which is
the complete intersection of a quadric and a cubic. In a similar way
as before, $Q_{6}$ has $15$ ordinary double points and $A_{1}$
and $A_{1}'$ are sent by $|L-3A_{1}|$ to two rational curves of
degree $6$ with intersection number $14$. 

For $k=4$, one get a degree $8$ model $Q_{8}$ of $X$ in $\PP^{5}$
which is the complete intersection of $3$ quadrics. That model has
$15$ ordinary double points and the curves $A_{1}$ and $A_{1}'$
are sent by $|L-4A_{1}|$ to two rational curves of degree $8$ with
intersection number $18$.  

\section{Nikulin configurations and automorphisms\label{sec:Nikulin-configurations-and}}

\subsection{Construction of an infinite order automorphism}

Let us denote by $K_{abcd}$ with $a,b,c,d\in\{0,1\}$ the $16$ $(-2)$-curves
on the K3 surface $X=\Km(A)$, and as before let $L$ be the polarization
coming from the polarization of $A$. \\
Let $K$ be the lattice generated by the following $16$ vectors $v_{1},\dots,v_{16}$:
\[
\begin{array}{c}
\frac{1}{2}\sum_{p\in A[2]}K_{p},\,\frac{1}{2}\sum_{W_{1}}K_{p},\,\frac{1}{2}\sum_{W_{2}}K_{p},\,\frac{1}{2}\sum_{W_{3}}K_{p},\,\frac{1}{2}\sum_{W_{4}}K_{p},\,K_{0000},\hfill\\
K_{1000},\,K_{0100},\,K_{0010},\,K_{0001},\,K_{0011},\,K_{0101},\,K_{1001},\,K_{0110},\,K_{1010},\,K_{1100}
\end{array}
\]
where $W_{i}=\{(a_{1},a_{2},a_{3},a_{4})\in(\ZZ/2\ZZ)^{4}\,|\,a_{i}=0\}$.
By results of Nikulin, \cite{Nikulin}, the lattice $K$ is the minimal
primitive sub-lattice of $H^{2}(X,\ZZ)$ containing the $(-2)$-curves
$K_{abcd}$. The discriminant group $K^{\vee}/K$ is isomorphic to
$(\ZZ_{2})^{6}$ and the discriminant form of $K$ is isometric to
the discriminant form of $U(2)^{\oplus3}$. 
\begin{lem}
(See \cite[Remark 2.3]{GS}) The Néron-Severi group $\NS(X)$ is generated
by $K$ and $v_{17}:=\frac{1}{2}(L+\omega_{4d})$, where $L$ is the
positive generator of $K^{\perp}$ with $L^{2}=4d$ (here $d=\frac{k(k+1)}{2}$),
and if $L^{2}=0$ mod $8$,
\[
\omega_{4d}=K_{0000}+K_{1000}+K_{0100}+K_{1100},
\]
if $L^{2}=4$ mod $8$,
\[
\omega_{4d}=K_{0001}+K_{0010}+K_{0011}+K_{1000}+K_{0100}+K_{1100}.
\]
\end{lem}
One has moreover
\begin{lem}
\label{lem:The-discriminant-group}(\cite[Remark 2.11]{GS}) The discriminant
group of $\NS(X)$ is isomorphic to $(\ZZ/2\ZZ)^{4}\times\ZZ/4d\ZZ$.
Suppose that $d=4$ mod $8$. Then $\NS(X)^{\vee}/\NS(X)$ is generated
by 
\[
\begin{array}{c}
w_{1}=\frac{1}{2}(v_{6}+v_{8}+v_{10}+v_{12}),\,\,\,w_{2}=\frac{1}{2}(v_{12}+v_{13}+v_{14}+v_{15}),\hfill\\
w_{3}=\frac{1}{2}(v_{11}+v_{13}+v_{14}+v_{16}),\,\,\,w_{4}=\frac{1}{2}(v_{9}+v_{10}+v_{12}+v_{13}),\hfill\\
w_{5}=\frac{1}{2}(v_{6}+v_{12}+v_{13})+\frac{1}{4d}(v_{7}+v_{8}+v_{9}+v_{10}+(1+2d)v_{11}+v_{16}-2v_{17})
\end{array}
\]
Suppose that $d=0$ mod $8$. Then $\NS(X)^{\vee}/\NS(X)$ is generated
by 
\[
\begin{array}{c}
w_{1}=\frac{1}{2}(v_{6}+v_{12}+v_{14}+v_{16}),\,\,\,w_{2}=\frac{1}{2}(v_{6}+v_{13}+v_{15}+v_{16}),\hfill\\
w_{3}=\frac{1}{2}(v_{6}+v_{8}+v_{10}+v_{12}),\,\,\,w_{4}=\frac{1}{2}(v_{6}+v_{8}+v_{9}+v_{13}),\hfill\\
w_{5}=\frac{1}{2}(v_{11}+v_{12}+v_{13})+\frac{1}{4d}((1+2d)v_{6}+v_{7}+v_{8}+v_{16}-2v_{17})
\end{array}
\]
In both cases, the discriminant form of $\NS(X)$ is isometric to
the discriminant form of $U(2)^{\oplus3}\oplus\la4d\ra$ and the transcendent
lattice $T_{X}=\NS(X)^{\perp}$ is isomorphic to $U(2)^{\oplus3}\oplus\la-4d\ra$.
\end{lem}
\begin{proof}
The columns of the inverse of the intersection matrix $(v_{i}v_{j})_{1\leq i,j\leq17}$
is a base of $\NS(X)^{\vee}$ in the base $v_{1},\dots,v_{17}$. From
that data we obtain the generators $w_{1},\dots,w_{5}$ of $\NS(X)^{\vee}/\NS(X)$.
The matrix $(w_{i}w_{j})_{1\leq i,j\leq5}$ is
\[
\left(\begin{array}{ccccc}
0 & \frac{1}{2} & 0 & 0 & 0\\
\frac{1}{2} & 0 & 0 & 0 & 0\\
0 & 0 & 0 & \frac{1}{2} & 0\\
0 & 0 & \frac{1}{2} & 0 & 0\\
0 & 0 & 0 & 0 & \frac{1}{4d}
\end{array}\right)\in M_{5}(\QQ/\ZZ),
\]
one has moreover $w_{i}^{2}=0$ mod $2\ZZ$ for $1\leq i\leq4$ and
$w_{5}^{2}=\frac{1}{4d}$ mod $2\ZZ$. Thus the discriminant form
\[
q:\NS(X)^{\vee}/\NS(X)\to\QQ/2\ZZ
\]
is isometric to the discriminant form of $U(2)^{\oplus3}\oplus\la4d\ra.$
Since $H^{2}(X,\ZZ)$ is unimodular, and $U(-2)\simeq U(2)$, we obtain
$T_{X}$ (for more details see e.g. \cite[Chap. 14, Proposition 0.2]{Huyb}).
\end{proof}
In Section \ref{sec1:Two-Nikulin-configurations}, we associated to
$L$ and to $A_{j}$ the divisors 
\[
L_{j}=(2k+1)L-2k(k+1)A_{j},\,\,A_{j}'=2L-(2k+1)A_{1}.
\]
The vector space endomorphism 
\[
\theta_{j}:\NS(X)\otimes\QQ\to\NS(X)\otimes\QQ
\]
defined by $\theta_{j}(A_{i})=A_{i}$ for $i\neq j$ and 
\[
\theta_{j}(A_{j})=A_{j}',\;\theta_{j}(L)=L_{j}
\]
is an involution, and we will see that it is an isometry (cf. Lemma
\ref{lem:The-morphisms isometries}). Let us define 
\[
\Phi_{1}=\theta_{2}\theta_{1}.
\]
The endomorphism $\Phi_{1}$ has infinite order, its characteristic
polynomial $\det(T\text{I}_{\text{d}}-\Phi_{1})$ is the product of
$(T-1)^{15}$ and the Salem polynomial 
\[
T^{2}+(2-4k^{2})T+1.
\]
The aim of this section is to prove the following result:
\begin{thm}
\label{thm:There-exists-an-invol}The automorphism $\Phi_{1}$ extends
to an effective Hodge isometry $\Phi$ of $H^{2}(X,\ZZ)$ and there
exists an automorphism $\iota$ of $X$ which acts on $H^{2}(X,\ZZ)$
by $\iota^{*}=\Phi$. 
\end{thm}
Let us start by the following Lemma: 
\begin{lem}
\label{lem:The-morphisms isometries}The morphisms $\t_{1},\,\t_{2},\,\Phi_{1}$
preserve $\NS(X)$ and are isometries of $\NS(X)$.
\end{lem}
\begin{proof}
It is simple to check that $\theta_{j}$ preserves the lattice generated
by $K,L$ and $v_{17}=\frac{1}{2}(L+\omega_{4d})$. Since for all
$1\leq i,j\leq16$ one has $\theta_{j}(A_{i})\theta_{j}(A_{k})=A_{i}A_{k}$,
$\theta_{j}(L)\theta_{j}(A_{i})=LA_{i}=0$, $\theta_{j}(L)^{2}=L^{2}$,
$\theta_{j}$ is an isometry of $\NS(X)$, hence so is $\Phi_{1}=\theta_{2}\theta_{1}$. 
\end{proof}
Let $T_{X}=\NS(X)^{\perp}$. We define $\Phi_{2}:T_{X}\to T_{X}$
as the identity. The map $(\Phi_{1},\Phi_{2})$ is an isometry of
$\NS(X)\oplus T_{X}$.
\begin{lem}
\label{lem:The-map-extends}The morphism $(\Phi_{1},\Phi_{2})$ extends
to an isometry $\Phi$ of $H^{2}(X,\ZZ)$.
\end{lem}
\begin{proof}
Let $L_{1},L_{2}$ be the lattices $L_{1}=\NS(X),$ $L_{2}=T_{X}=\NS(X)^{\perp}.$
Let us denote by 
\[
q_{i}:L_{i}^{\vee}/L_{i}\to\QQ/2\ZZ
\]
the discriminant form of $L_{i}$. By Lemma \ref{lem:The-discriminant-group}
and its proof, we know the form $q_{1}$ on the base $w_{i}$. \\
One has $L_{2}=U(2)\oplus U(2)\oplus\la-4d\ra$. Let us take the base
$e_{i},\,1\leq i\leq5$ of $L_{2}$ such that the intersection matrix
of the $e_{j}$'s is 
\[
(e_{i}e_{j})_{1\leq i,j\leq5}=-\left(\begin{array}{ccccc}
0 & 2 & 0 & 0 & 0\\
2 & 0 & 0 & 0 & 0\\
0 & 0 & 0 & 2 & 0\\
0 & 0 & 2 & 0 & 0\\
0 & 0 & 0 & 0 & 4d
\end{array}\right).
\]
The elements $w_{i}'=\frac{1}{2}e_{i}$ for $1\leq i\leq4$ and $w_{5}'=\frac{1}{4d}e_{5}$
are generators of $L_{2}^{\vee}/L_{2}$. Let 
\[
\phi:L_{2}^{\vee}/L_{2}\to L_{1}^{\vee}/L_{1}
\]
be the isomorphism (called the gluing map) defined by 
\[
\phi(w_{i}')=w_{i}.
\]
One has $q_{1}(\phi(\sum a_{i}w_{i}'))=-q_{2}(\sum a_{i}w_{i}')$
i.e. 
\[
q_{2}=-\phi^{*}q_{1}.
\]
Since $L_{1},L_{2}$ are primitive sub-lattices of the even unimodular
lattice $H^{2}(X,\ZZ)$ with $L_{2}=L_{1}^{\perp}$, the lattice $H^{2}(X,\ZZ)$
is obtained by gluing $L_{1}$ with $L_{2}$ by the gluing isomorphism
$\phi$. In other words $H^{2}(X,\ZZ)$ is generated by all the lifts
in $L_{1}^{\vee}\oplus L_{2}^{\vee}$ of the elements $(w_{i},w_{i}')$,
$i=1,\dots,5$ of the discriminant group of $L_{1}\oplus L_{2}$.
\\
According to general results (see e.g. \cite[Page 5]{Mc}), the element
$(\Phi_{1},\Phi_{2})$ of the orthogonal group of $L_{1}\oplus L_{2}$
extends to $H^{2}(X,\ZZ)$ if and only if the gluing map $\phi$ satisfies
$\phi\circ\Phi_{2}=\Phi_{1}\circ\phi$. A simple computation gives
that for $1\leq i\leq4$, one has $\theta_{j}w_{i}=-w_{i}=w_{i}$
(for $j\in\{1,2\}$), thus $\Phi_{1}(w_{i})=w_{i}$. Moreover we compute
that
\[
\theta_{j}(w_{5})=(1-2k^{2})w_{5}
\]
and since $(1-2k^{2})^{2}=1$ modulo $4d=2k(k+1)$, one gets $\Phi_{1}(w_{5})=\theta_{2}\theta_{1}w_{5}=\omega_{5}$.
Since by definition $\Phi_{2}(w_{i}')=w_{i}'$ for $i=1,\dots,5$,
we obtain the desired relation $\phi\circ\Phi_{2}=\Phi_{1}\circ\phi$.
\end{proof}
\begin{rem}
\label{rem15:Because-of-the}Because of the relation $\theta_{j}(w_{5})=(1-2k^{2})w_{5}$,
$j\in\{1,2\}$ at the end of the proof of Lemma \ref{lem:The-map-extends},
it is not possible to extend the involution $\theta_{j}$ to an isometry,
unless $k=1$. In that case, using the proof of Lemma \ref{lem:The-Hodge-isometry is effective}
below, the involution $\theta_{j}$ extends to an effective Hodge
isometry (with action by multiplication by $-1$ on $T_{X}$). The
resulting non-symplectic involution is in fact known under the name
of projection involution, see e.g. \cite{Keum}. 
\end{rem}
\begin{lem}
The morphism $\Phi$ is an Hodge isometry: its $\CC$-linear extension
$\Phi_{\CC}:H^{2}(X,\CC)\to H^{2}(X,\CC)$ preserves the Hodge decomposition. 
\end{lem}
\begin{proof}
The map $\Phi$ is the identity on the space $T_{X}\otimes\CC$ containing
the period.
\end{proof}
\begin{lem}
\label{lem:The-Hodge-isometry is effective}The Hodge isometry $\Phi$
is effective.
\end{lem}
\begin{proof}
Since $X$ is projective by \cite[Proposition 3.11]{BPVdV}, it is
enough to prove that the image by $\Phi$ of one ample class is an
ample class. Let $m\geq2$ be an integer. By \cite[Proposition 4.3]{GS},
the divisor $D=mL-\frac{1}{2}\sum_{i\geq1}A_{i}$ is ample. The image
by $\theta_{1}$ of $D$ is 
\[
\theta_{1}(D)=mL_{1}-\frac{1}{2}\left(A_{1}'+\sum_{i\geq2}A_{i}\right)
\]
where by Section \ref{sec1:Two-Nikulin-configurations} we have that
$A_{1}'$ is a $(-2)$-curve, which is disjoint from the $A_{j},\,j\geq2$,
and these $16$ $(-2)$-curves have intersection $0$ with $L_{1}=\theta_{1}(L)$.
There exists an Abelian surface $B'$ such that $X=\Km(B')$ and these
$16$ $(-2)$-curves are resolution of the $16$ singularities in
$B'/[-1]$. Moreover $L_{1}$ comes from a polarization $M'$ on $B'$,
which clearly generates $\NS(B')$. Thus again by \cite[Proposition 4.3]{GS},
$\theta_{1}(D)$ is ample. \\
The analogous proof with $(\theta_{2},\,A_{2})$ instead of $(\t_{1},\,A_{1})$
gives us that $\theta_{2}(D)$ is also ample. Since $\theta_{i},\,i=1,2$
are involutions and $\Phi=\theta_{2}\theta_{1}$, we conclude that
\[
\Phi(\theta_{1}(D))=\theta_{2}(D)
\]
is ample, and thus $\Phi$ is effective.
\end{proof}

We can now apply the Torelli Theorem for K3 surfaces (see \cite[Chap. VIII, Theorem 11.1]{BPVdV}):
since $\Phi$ is an effective Hodge isometry there exists an automorphism
$\iota:X\to X$ such that $\iota^{*}=\Phi$. This finishes the proof
of Theorem \ref{thm:There-exists-an-invol}. \qed
\begin{rem}
The Lefschetz formula for the fixed locus $X^{\iota}$ of $\iota$
on $X$ gives 
\[
\chi(X^{\iota})=\sum_{i=0}^{4}(-1)^{i}tr(\Phi|H^{i}(X,\RR))=1+(4k^{2}+18)+1=20+4k^{2},
\]
(here $\iota^{*}=\Phi$). If $k=1$ then $\chi(X^{\iota})=24$ and
we can easily see that $X^{\iota}$ contains two rational curves.
Indeed in this case as remarked before (Remark \ref{rem15:Because-of-the})
$\theta_{i}$, $i=1,2$ can be extended to a non-symplectic involution
(still denoted $\t_{i}$) of the whole lattice $H^{2}(X,\ZZ)$. The
fixed locus of each $\theta_{i}$, $i=1,2$ are the curves pull-back
on $X$ of the six lines in the branching locus of the double cover
of $\PP^{2}$ (the $\theta_{i}$, $i=1,2$ are the covering involutions).
These curves are different except for the pull-backs $\ell_{1}$ and
$\ell_{2}$ of two lines, which are the lines passing through the
point of the branching curve corresponding to $A_{2}$ if we consider
the double cover determined by the involution $\theta_{1}$, respectively
through the point corresponding to $A_{1}$ if we consider $\theta_{2}$.
So the infinite order automorphism $\iota$ corresponding to $\Phi=\theta_{2}\theta_{1}$
fixes the two rational curves $\ell_{1}$ and $\ell_{2}$ on $X$.
By using results of Nikulin on non-symplectic involutions \cite{NikiPezzo}
the invariant sublattices $H^{2}(X,\ZZ)$ for the action of $\theta_{i}$,
$i=1,2$ are both isometric to $U\oplus E_{8}(-1)\oplus\langle-2\rangle^{\oplus6}$. 
\end{rem}

\subsection{\label{subsec:Some-remarks-on Auto}Action of the automorphism group
on Nikulin configurations}

The aim of this sub-section is to prove the following result
\begin{thm}
\label{thm:no automorphisms}Suppose that $k\geq2$. There is no automorphism
$f$ of $X$ sending the configuration $\mathcal{C}=\sum_{i=1}^{16}A_{i}$
to the configuration $\mathcal{C}'=A_{1}'+\sum_{i=2}^{16}A_{i}$. 
\end{thm}
Suppose that such an automorphism $f$ exists. The group of translations
by the $2$-torsion points on $B$ acts on $X=\Km(B)$ and that action
is transitive on the set of curves $A_{1},\dots,A_{16}$. Thus up
to changing $f$ by $f\circ t$ (where $t$ is such a translation),
one can suppose that the image of $A_{1}$ is $A_{1}'$. Then the
automorphism $f$ induces a permutation of the curves $A_{2},\dots,A_{16}$.
The $(-2)$-curve $A_{1}''=f^{2}(A_{1})=f(A_{1}')$ is orthogonal
to the $15$ curves $A_{i},\,i>1$ and therefore its class is in the
group generated by $L$ and $A_{1}$. By the description of $\NS(X)$,
the $(-2)$-class $A_{1}''=aA_{1}+bL$ has coefficients $a,b\in\ZZ$
. Moreover $a,b$ satisfy the Pell-Fermat equation
\begin{equation}
a^{2}-k(k+1)b^{2}=1.\label{eq:Pell-Fermat-1}
\end{equation}
Let us prove:
\begin{lem}
Let $C=aA_{1}+bL$ be an effective $(-2)$-class. Then there exists
$u,v\in\NN$ such that $aA_{1}+bL=uA_{1}+vA_{1}'$, in particular
the only $(-2)$-curves in the lattice generated by $L$ and $A_{1}$
are $A_{1}$ and $A_{1}'$.
\end{lem}
\begin{proof}
If $(a,b)$ is a solution of equation \eqref{eq:Pell-Fermat-1}, then
so are $(\pm a,\pm b)$. We say that a solution is positive if $a\geq0$
and $b\geq0$. Let us identify $\ZZ^{2}$ with $A=\ZZ[\sqrt{N}]$
by sending $(a,b)$ to $a+b\sqrt{N}$, where $N=k(k+1)$. The solutions
of \eqref{eq:Pell-Fermat-1} are units of the ring $A$. According to
the Chakravala method solving equation \eqref{eq:Pell-Fermat-1},
there exists a solution $\a+\b\sqrt{N}$ (called fundamental) with
$\a,\b\in\NN^{*}$ such that the positive solutions are the elements
of the form
\[
a_{m}+b_{m}\sqrt{N}=(\a+\b\sqrt{N})^{m},\,m\in\NN.
\]
The first term of the sequence of convergents of the regular continued
fraction for $\sqrt{N}$ is 
\[
\frac{2k+1}{2},
\]
and since $(2k+1,2)$ is a solution of \eqref{eq:Pell-Fermat-1},
the fundamental solution is $(\a,\b)=(2k+1,2)$. \\
An effective $(-2)$-class $C=aA_{1}+bL$ either equals $A_{1}$ or
satisfies $CL>0$ and $CA_{1}>0$, therefore $b>0$ and $a<0$. Thus
if $C\neq A_{1}$, there exists $m$ such that $C=-a_{m}A_{1}+b_{m}L$.
Since $A_{1}'=2L-(2k+1)A_{1}$, one obtains
\[
-a_{m}A_{1}+b_{m}L=\frac{b_{m}}{2}A_{1}'+((2k+1)\frac{b_{m}}{2}-a_{m})A_{1}
\]
and the Lemma is proved if the coefficients $u_{m}=\frac{b_{m}}{2}$
and $v_{m}=(2k+1)\frac{b_{m}}{2}-a_{m}$ are both positive and in
$\ZZ$. Using the relation
\[
a_{m+1}+b_{m+1}\sqrt{N}=(2k+1+2\sqrt{N})(a_{m}+b_{m}\sqrt{N}),
\]
that follows from an easy induction.
\end{proof}
Therefore we conclude that $A_{1}''=A_{1}$ i.e. $f$ permutes $A_{1}$
and $A_{1}'$. Let us finish the proof of Theorem \ref{thm:no automorphisms}:
\begin{proof}
The class $f^{*}L$ is orthogonal to $A_{1}',A_{2},\dots,A_{16}$,
thus this is a multiple of the class $L'=(2k+1)L-2k(k+1)A_{1}$ which
has the same property. Since both classes have the same self-intersection
and are effective, one gets $f^{*}L=L'$; by the same reasoning, since
$f^{*}A_{1}'=A_{1}$, one gets $f^{*}L'=L$.  By \cite[Proposition 4.3]{GS},
the divisor 
\[
D=2L-\frac{1}{2}\sum_{i\geq1}A_{i}
\]
is ample, thus $f^{*}D=2L'-\frac{1}{2}(A_{1}'+\sum_{i\geq2}A_{i})$
is also ample and so is $D+f^{*}D$. Moreover $D+f^{*}D$ is preserved
by $f$, thus by \cite[Proposition 5.3.3]{Huyb}, the automorphism
$f$ has finite order. Up to taking a power of it, one can suppose
that $f$ has order $2^{m}$ for some $m\in\NN^{*}$. Suppose $m=1$,
ie $f$ is an involution. Then 
\[
\frac{1}{2}(A_{1}+A_{1}')=L-kA_{1}
\]
is fixed, there are curves $A_{i},\,i>1$ such that $f(A_{i})=A_{i}$
(say $s$ of such curves; necessarily $s$ is odd) and $f$ permutes
the remaining curves $A_{j}$ by pairs (there are $t=\frac{1}{2}(15-s)$
such pairs). Let $\G$ be the lattice generated by the classes $A_{i}$
fixed by $f$, by $A_{j}+f(A_{j})$ if $f(A_{j})\neq A_{j}$ and by
$L-kA_{1}$. It is a finite index sub-lattice of $\NS(X)^{f}$, the
fix sub-lattice of the Néron-Severi group. The discriminant group
of $\G$ is 
\[
\ZZ/2k\ZZ\times(\ZZ/2\ZZ)^{s}\times(\ZZ/4\ZZ)^{t}.
\]
Since in $\NS(X)$ there is at most a coefficient $\frac{1}{2}$ on
$L$, the discriminant of $\NS(X)^{f}$ contains $\ZZ/k\ZZ$. If $f$
was non-symplectic, then $\mathcal M=\NS(X)^{f}$ would be a $2$-elementary
lattice (see \cite{AST}; it means that the discriminant group $\mathcal{M}^{*}/\mathcal{M}\simeq(\ZZ/2\ZZ)^{a}$
for some integer positive $a$). But for $k>2$ this is impossible,
therefore $f$ has to be symplectic. \\
For $k=2$, we use the model $Y\hookrightarrow\PP^{3}$ of degree
$4$ with $15$ nodes of $X$ determined by the divisor $L-2A_{1}$.
Since $f$ preserves $L-kA_{1}$, the involution on $X$ induces an
involution (still denoted $f$) on $\PP^{3}=|L-kA_{1}|$ preserving
$Y$. Up to conjugation, $f$ is $x\to(-x_{1}:x_{2}:x_{3}:x_{4})$
or $x\to(-x_{1}:-x_{2}:x_{3}:x_{4})$. \\
Suppose that $f$ is $f:x\to(-x_{1}:x_{2}:x_{3}:x_{4})$. The hyperplane
$x_{1}=0$ cuts the quartic $Y$ into a quartic plane curve $C_{0}\hookrightarrow Y$.
The surface $Y$ is a double cover of $\PP(2,1,1,1)$ branched over
$C_{0}\hookrightarrow\PP(2,1,1,1)$. The quartic $C_{0}$ is irreducible
and reduced, since otherwise $X$ would have Picard number $>17$.
The singularities on $C_{0}$ are at most nodes and the corresponding
nodes on $Y$ are fixed by $f$. Let us recall that the number $s$
of fixed nodes is odd. \\
Suppose that $C_{0}$ contains $3$ nodes. Its pull back $C_{0}'$
on $X$ is a smooth rational curve. The rank of the sub-lattice $\NS(X)^{f}$
is $1+s+t=10$. By \cite[Figure 1]{AST}, the genus of the fixed curve
$C_{0}'$ must be strictly positive, which is a contradiction. \\
Suppose that $C_{0}$ contains $2$ nodes, then the isolated fixed
point $(1:0:0:0)$ is also a node; the rank of $\NS(X)^{f}$ is still
$10$. One has 
\[
[\NS(X)^{f}:\Gamma]^{2}=\frac{\det\G}{\det\NS(X)^{f}}=\frac{2^{2+1+2t}}{2^{a}}=2^{17-a},
\]
thus $a$ is odd. However by \cite[Figure 1]{AST}, when $\NS(X)^{f}$
has rank $10$, the integer $a$ is always even, this is a contradiction.\\
Suppose that $C_{0}$ contains $1$ node. Its pull back on $X$ is
a smooth genus $2$ curve. One has $\rk\NS(X)^{f}=9$. By \cite[Figure 1]{AST},
since the fixed curve has genus $2$, one has $a=9$, therefore 
\[
[\NS(X)^{f}:\Gamma]^{2}=2^{17-a}=2^{8},
\]
and there are at most $4$-classes which are $2$-divisible in the
discriminant group 
\[
\ZZ/4\ZZ\times\ZZ/2\ZZ\times(\ZZ/4\ZZ)^{7}
\]
of $\G$. But then the discriminant group of $\NS(X)^{f}$ would contain
a sub-group $\ZZ/4\ZZ$, which is a contradiction. \\
Suppose that $f$ is $f:x\to(-x_{1}:-x_{2}:x_{3}:x_{4})$ (observe
that we can not exclude immediately this case since $Y$ is singular.
If $Y$ would be smooth then such an $f$ would correspond to a symplectic
automorphism). The line $x_{1}=x_{2}=0$ or $x_{3}=x_{4}=0$ cannot
be included in $Y,$ otherwise $Y$ would be singular along that line
(this is seen using the equation of $Y$). The number of fixed nodes
being odd, there are $1$ or $3$ fixed nodes of $Y$ on these two
lines (the intersection number of each lines with $Y$ being $4$).
\\
Suppose that one node is fixed. The corresponding $(-2)$-curve on
$X$ must be stable, moreover $\rk\NS(X)^{f}=9$. But by \cite[Figure 1]{AST},
there is no non-symplectic involution on a K3 such that $\rk\NS(X)^{f}=9$
and the fix-locus is a $(-2)$-curve or is empty. By the same reasoning,
one can discard the case of $3$ stable rational curves. \\
We therefore proved that for any $k>1$, $f$ must be symplectic.

A symplectic automorphism acts trivially on the transcendental lattice
$T_{X}$, which in our situation has rank $5$. Therefore the trace
of $f$ on $H^{2}(X,\ZZ)$ equals $6+s>6$. But the trace of a symplectic
involution equals $6$ (see e.g. \cite[Section 1.2]{SvG}). This is
a contradiction, thus $f$ cannot have order $2$ and $m$ is larger
than $1$.

The automorphism $g=f^{2^{m-1}}$ has order $2$ and $g(A_{1})=A_{1},\,g(A_{1}')=A_{1}'$,
thus $g(L)=L$. There are curves $A_{i},\,i>1$ such that $f(A_{i})=A_{i}$
(say $s$ of such, $s$ is odd since $A_{1}$ is fixed) and the remaining
curves $A_{j}$ are permuted $2$ by $2$ (there are $t=\frac{1}{2}(15-s)$
such pairs). Let similarly as above $\G'$ be the sub-lattice generated
by $L,A_{1}$ and the fix classes $A_{i}$, $A_{j}+g(A_{j})$. It
is a finite index sub-lattice of $\NS(X)^{g}$ and its discriminant
group is 
\[
\ZZ/2k(k+1)\ZZ\times(\ZZ/2\ZZ)^{s+1}\times(\ZZ/4\ZZ)^{t}.
\]
By the same reasoning as before, the automorphism $g$ must be symplectic
as soon as $k>1$. But the trace of $g$ is $8+s>6$, thus $g$ cannot
be symplectic either. Therefore we conclude that such an automorphism
$f$ does not exist. 
\end{proof}

\subsection{Consequences on the Kummer structures on $X$\label{subsec:Nikulin-structures-and Fermat}}

A Kummer structure on a K3 surface $X$ is an isomorphism class of
Abelian surfaces $B$ such that $X\simeq\Km(B)$.  The following
Proposition is stated in \cite{HLOY}; we give here a proof for completeness:
\begin{prop}
\label{prop:The-Kummer-structures}The Kummer structures on $X$ are
in one-to-one correspondence with the orbits of Nikulin configurations
under the automorphism group $\aut(X)$ of $X$.
\end{prop}
\begin{proof}
Let $\mathcal{C}$ be a Nikulin configuration on the K3 surface $X$. By \cite[Theorem 1]{Nikulin} of Nikulin, there exists a unique (up to isomorphism) double cover $\tilde B \to X$ branched over $\mathcal{C}$. Moreover the minimal model $B$ of $\tilde B$ is an Abelian surface, and $X$ is the Kummer surface associated to $B$, $\mathcal C$ being the union of the exceptional curves of the resolution $X=\Km(B)\to B/[-1]$.

Let $\mu:X\to X$ be an automorphism sending a Nikulin configuration
$\mathcal{C}$ to $\mathcal{C}'$. 
Let $B$, $B'$ be the abelian
surfaces such that $\mathcal{C}$ (resp. $\mathcal{C}'$) is the configuration
associated to $\Km(B)=X$ (resp. $\Km(B')=X$). \\
%Let us prove that $B$ is isomorphic to $B'$. 
Let $\tilde{B}\to B$ and $\tilde{B}'\to B'$ be the blow-up at the sixteen $2$-torsion
points of $B$ (resp. $B'$). Consider the natural map $\tilde{B}\to X\stackrel{\mu}{\to}X$:
it is a double cover of $X$ branched over $\mathcal{C}'$ and ramified
over the exceptional locus of $\tilde{B}\to B$, thus by the results of Nikulin we just recalled,
 $\tilde{B}$ is isomorphic to $\tilde{B}'$ and $B\simeq B'$.\\
Reciprocally, suppose that there is an isomorphism $\phi:B\to B'$.
It induces an isomorphism $\tilde{\phi}:\tilde{B}\to\tilde{B}'$ that
induces an isomorphism $X=\Km(B)\to\Km(B')=X$ which sends the Nikulin
configuration $\mathcal{C}$ corresponding to $B$ to the Kummer structure
$\mathcal{C}'$ corresponding to $B'$. %(observe that the $[-1]$-maps commute with $\tilde{\phi}$). 
\end{proof}
According to \cite{HLOY}, the number of Kummer structures is finite.
If $X=\Km(B)$ and $B^{*}$ is the dual of $B$, by result of Gritsenko
and Hulek \cite{GH} one has also $X\simeq\Km(B^{*})$, thus if $B$
is not principally polarized, the number of Kummer structures is
at least $2$. \\
When $\NS(B)=\ZZ M$, by results of Orlov \cite{Orlov} on derived
categories, the number of Kummer structures equals $2^{s}$ where
$s$ is the number of prime divisor of $\frac{1}{2}M^{2}$. In our
situation one has $M^{2}=k(k+1)$. By subsection \ref{subsec:Some-remarks-on Auto}
as soon as $k>2$, there is no automorphism sending the configuration
$\mathcal{C}=\sum_{i=1}^{16}A_{i}$ to $\mathcal{C}'=A_{1}'+\sum_{i=2}^{16}A_{i}$,
thus
\begin{cor}
Suppose $k\geq2$. The two Nikulin configurations $\mathcal{C}=\sum_{i=1}^{16}A_{i}$
and $\mathcal{C}'=A_{1}'+\sum_{i=2}^{16}A_{i}$ represent two distinct
Kummer structures on $X$.
\end{cor}
\begin{rem}
When $k=2$ then $\frac{k(k+1)}{2}=3$ is divisible by one prime,
thus the configurations $\mathcal{C}$ and $\mathcal{C}'$ are the
two representatives of the set of Kummer structures on $X=\Km(B)$.
Observe that $X$ is also isomorphic to $\Km(B^{*})$, where $B^{*}$
is the dual of $B$. Since $B$ is not isomorphic to $B^{*}$, the
double cover of $X$ branched over $\mathcal{C}'$ is (the blow-up
of) $B^{*}$.
\end{rem}

\section{bi-double covers associated to Nikulin configurations\label{sec:bi-double-covers-associated}}

\subsection{A hyperelliptic curve with genus $\leq2k$ and a point of multiplicity
$2(2k+1)$ on the Abelian surface $B$\label{subsec:An-hyperelliptic-curve}}

We keep the notations as above: $(B,M)$ is a polarized Abelian variety
with $M^{2}=k(k+1)$ and $\Pic(B)=\ZZ M$. The associated K3 surface
$X=\Km(B)$ contains the $17$ smooth rational curves 
\[
A_{1},A_{1}',\,A_{2},\dots,A_{16}
\]
such that $A_{1},\dots,A_{16}$ are the $16$ disjoint $(-2)$-curves
arising from the Kummer structure, $A_{1}'$ is a $(-2)$-curve such
that $A_{1}',\,A_{2},\dots,A_{16}$ is a Nikulin configuration and
\[
A_{1}A_{1}'=4k+2.
\]
Let $\pi:\tilde{B}\to B$ be the blow-up of $B$ at the $16$ points
of $2$-torsion, so that there is a natural double cover $\tilde{B}\to X=\Km(B)$
branched over the $16$ exceptional divisors. 

Let $\tilde{\G}$ be the pull-back of $A_{1}'$ on $\tilde{B}$ and
let $\G$ be the image of $\tilde{\G}$ on $B$. We denote by $E\hookrightarrow\tilde{B}$
the $(-1)$-curve above $A_{1}$. Let us prove the following result
\begin{prop}
\label{prop:The-curve-is hyperelliptic}The curve $\G\hookrightarrow B$
is hyperelliptic, it has geometric genus $\leq2k$ and has a unique
singularity, which is a point of multiplicity $2(2k+1)$. The curve
$\Gamma$ is in the linear system $|4M|$, in particular $\G^{2}=16k(k+1)$.
\end{prop}
\begin{proof}
The singularities
on a curve that is the union of two smooth curves on a smooth surface are of type  
\[
\qa_{2m-1},\,m\geq1,
\]
where an equation of an $\qa_{2m-1}$ singularity is $\{x^{2m}-y^{2}=0\}$.
This is well-known by experts but we couldn't find a reference and we therefore sketch a proof. At a singularity $p$, there are local parameters
 $x,y$ such that $C_1$ is given by $y=0$. By the implicit function theorem, 
we reduce to the case where the curve $C_2$ has equation $y=x^m$ for some $m>0$. 
 Then the singularity has equation $\{y(y-x^m)=0\}$, which after a variable change becomes $\{x^{2m}-y^{2}=0\}$.\\
Let us denote by $\a_{m}$ the number of $\qa_{2m-1}$ singularities
on the union $A_{1}+A_{1}'$. Since a $\qa_{2m-1}$ singularity contributes
to $m$ in the intersection of $A_{1}$ and $A_{1}'$, one has 
\[
\sum_{m\geq0}m\a_{m}=4k+2.
\]
By \cite[Table 1, Page 109]{BPVdV}, the curve $\tilde{\G}\hookrightarrow\tilde{B}$
has a singularity $\qa_{m-1}$ above a singularity $\qa_{2m-1}$ of
$A_{1}+A_{1}'$ (by abuse of language a $\qa_{0}$-singularity means
a smooth point). Let $\Gamma'$ be de normalization of $\tilde{\G}$;
a $\qa_{2m-1}$-singularity contributes in the ramification locus
of the double cover $\G'\to A_{1}$ (induced by $\tilde{\G}\to A_{1}$)
by $1$ if $m$ is odd and $0$ if $m$ is even. Therefore the geometric
genus of $\G$ is 
\[
2g(\G)-2=2\cdot(-2)+\sum_{m\,odd}\a_{m}\leq4k+2,
\]
which gives $g(\G)\leq2k$. The singularities of $\tilde{\G}$ are
at its intersection with $E$, and since 
\[
\tilde{\G}E=\frac{1}{2}\pi_{1}^{*}A_{1}\pi_{1}^{*}A_{1}'=A_{1}A_{1}',
\]
we obtain $\tilde{\G}E=4k+2$. Since $E$ is contracted by the map
$\tilde{B}\to B$, the curve $\G$ (image of $\tilde{\G}$) has a
unique singular point of multiplicity $4k+2$. \\
Since $A_{1}'=2L-(2k+1)A_{1}$, its pull back on $\tilde{B}_{1}$
is $4\tilde{M}-2(2k+1)\tilde{\G}$ and its image $\Gamma$ has class
$4M$, thus $\G^{2}=16k(k+1)$. 
\end{proof}
\begin{rem}
Let us choose the point of multiplicity $2(2k+1)$ of $\G$ as the
origin $0$ of the group $B$. By construction the curve $\G$ does
not contain any non-trivial $2$-torsion point of $B_{1}$.
\end{rem}

\subsubsection*{The problem of the intersection of $A_{1}$ and $A_{1}'$}

It is a difficult question to understand how the curves $A_{1}$ and
$A_{1}'$ intersect on the Kummer surface $X=\Km(B)$. For $k=1$
and $2$ we know that these curves intersects transversally in $4k+2$
points, and thus $g(\Gamma)=2k$. For $k=1$, it follows from the
geometric description of the Jacobian Kummer surface as a double cover
of the plane branched over $6$ line. For $k=2$ it is a by-product
of \cite{RRS}. 

In \cite[Section 5, pp. 54--56]{BOPY} Bryan, Oberdieck, Pandharipande
and Yin, quoting results of Graber, discuss on a related problem which
is about hyperelliptic curves on Abelian surfaces. Let $f:C\to B$
be a degree $1$ morphism from a hyperelliptic curve $C$ to an Abelian
surface $B$ with image $\bar{C}$, such that the polarization $[\bar{C}]$
is generic. Let $\iota:C\to C$ be the hyperelliptic involution. 
\begin{conjecture}
\label{conj:(see-)}(see \cite{BOPY}) Suppose $B$ generic among
polarized Abelian surfaces. The differential of $f$ is injective
at the Weierstrass points of $C$, and no non-Weierstrass points $p$
is such that $f(p)=f(\iota(p))$. 
\end{conjecture}
In our situation, that Conjecture means that the rational curves $A_{1}$
and $A_{1}'$ meet transversally. Indeed if they meet at a point tangentially
with order $m\geq2$, then the curve above $A_{1}'$ has a $\qa_{m-1}$
singularity. If $m$ is even, there is no branch points above that
singular point, and thus there are points $p,\iota(p)$ (with $p$
non-Weierstrass) which are mapped to the same point by $f$. If $m$
is odd and $>1$, then the curve $C$ above $A_{1}'$ has a singularity
$\qa_{m-1}$ of type ``cusp'', the differential of its normalization
is $0$. 

Construction of (nodal or smooth) rational curves on K3 surfaces is
an important problem, see e.g. \cite[Chapter 13]{Huyb} for a discussion.
The existence of two smooth rational curves $C_{1},C_{2}$ intersecting
transversely and such that $C_{1}+C_{2}$ is a multiple $nH$ of a
polarization $H$ is also a key point for obtaining the existence
of an integer $n$ such that there exists an integral rational curve
in $|nH|$, see \cite[Chapter 13, Theorem 1.1]{Huyb} and its proof.

\subsection{Invariants of the bidouble covers associated to the special configuration\label{subsec:Invariants-of-the}}

Let us define 
\[
D_{1}=A_{1}',\,D_{2}=A_{1},\,D_{3}=\sum_{i=2}^{16}A_{j}.
\]
By Nikulin results, the divisors $\sum_{i=2}^{16}A_{j}+A_{1}$ and
$\sum_{i=2}^{16}A_{j}+A_{1}'$ are $2$-divisible and therefore there
exists $L_{1},L_{2},L_{3}$ such that 
\[
2L_{i}=D_{j}+D_{k}
\]
for $\{i,j,k\}=\{1,2,3\}$. Each $L_{i}$ defines a double cover 
\[
\pi_{i}:\tilde{B_{i}}\to X
\]
branched over $D_{j}+D_{k}$ (here $\tilde{B}_{1}=\tilde{B}$). For
$i=1,2$, above the $16$ $(-2)$-curves of the branch locus of $\pi_{i}:\tilde{B_{i}}\to X$
there are $16$ $(-1)$-curves. Let $\tilde{B}_{i}\to B_{i}$ be the
contraction map, so that the surface $B_{i}$ ($i=1,2$) is an Abelian
surface. \\
The divisors $D_{i},L_{i},\,i\in\{1,2,3\}$ are the data of a bi-double
cover 
\[
\pi:V\to X
\]
which is a $(\ZZ/2\ZZ)^{2}$-Galois cover of $X$ branched over the
curves $A_{1}',\,A_{i},\,i\geq1$. By classical formulas, the surface
$V$ has invariants
\[
\begin{array}{c}
\chi(O_{V})=4\cdot2+\frac{1}{2}\sum L_{i}^{2}=k\\
K_{V}^{2}=(\sum L_{i})^{2}=8k-30.\quad
\end{array}
\]
The surface $V$ contains $30$ $(-1)$-curves, which are above the
$15$ curves $A_{i},\,i>1$. The surface $V$ is smooth if and only
if the intersection of $A_{1}$ and $A_{1}'$ is transverse, i.e.
if Conjecture \ref{conj:(see-)} holds. Let us suppose that this is
indeed the case, then one has moreover the formula 
\[
p_{g}(V)=p_{g}(X)+\sum h^{0}(X,L_{i}).
\]
The space $H^{0}(X,L_{i})$ is $0$ for $i=1,2$ because the double
covers branched over $D_{2}+D_{3}$ or $D_{1}+D_{3}$ are Abelian
surfaces $B_{i}$ ($i=1,2$) and $1=p_{g}(B_{i})=p_{g}(X)+h^{0}(X,L_{i})\geq1$.
It remains to compute $h^{0}(X,L_{3})$. The divisor $L_{3}=A_{1}+A_{1}'$
is big and nef (see section \ref{sec1:Two-Nikulin-configurations}).
By Riemann-Roch, one has
\[
\chi(L_{3})=\frac{1}{2}L_{3}^{2}+2=k+2.
\]
By Serre duality and Mumford vanishing Theorem, $h^{1}(L_{3})=h^{1}(L_{3}^{-1})=0$.
Moreover $h^{2}(L_{3})=h^{0}(-L_{3})=0$, thus $h^{0}(L_{3})=k+2$
and therefore $p_{g}(V)=k+3$. Let $V\to Z$ be the blow-down map
of the $30$ $(-1)$-curves on $V$ which are above the $15$ $(-2)$-curves
$A_{i},\,i>1$ in $X$. We thus obtain:
\begin{prop}
Suppose that $A_{1}$ and $A_{1}'$ intersect transversally. The surface
$Z$ has general type and its invariants are 
\[
\chi=k,\,K_{Z}^{2}=8k,\,p_{g}(Z)=k+3,\text{ and }q=4.
\]
\end{prop}
The surface $Z$ is minimal as we see by using the rational map of
$Z$ onto the Abelian surface $B_{1}$.
\begin{rem}
The surface $Z$ satisfies 
\[
c_{1}^{2}=2c_{2}=8k.
\]
Among surfaces with $c_{1}^{2}=2c_{2}$ there are surfaces whose universal
covers is the bi-disk $\HH\times\HH$. For $k=1$, it turns out that
$Z$ is the product of two genus $2$ curves, thus its universal cover
is $\HH\times\HH$. For $k=2$, we obtain the so-called Schoen surfaces,
whose universal cover is not $\HH\times\HH$ (see \cite{CMR}, \cite{RRS}). 
\end{rem}
Let $(W,\omega)$ be a smooth projective algebraic variety of dimension
$2n$ over $\CC$ equipped with a holomorphic $(2,0)$-form of maximal
rank $2n$. Let us recall that a $n$ dimensional subvariety $Z\subset W$
is called Lagrangian if the restriction of $\omega$ to $Z$ is trivial.
We remark that
\begin{prop}
The surface $Z$ is a Lagrangian surface in $B_{1}\times B_{2}$.
\end{prop}
\begin{proof}
In \cite{BT}, Bogomolov and Tschinkel associate a Lagrangian surface
to the data of Kummer surfaces $S_{1}=\Km(A_{1}),S_{2}=\Km(A_{2})$
and a $K3$ surface $S$ such that there is a rational map $S\to S_{i},$
$i=1,2$. \\
In our situation, we take $S_{1}=S_{2}=S=\Km(B)$, we consider the
Kummer structure $\Km(B_{1})$ for $S_{1}$ and the Kummer structure
$\Km(B_{2})$ (see also Remark \ref{rem:To-the-pair}) for $S_{2}$,
and the identity map for $S\to S_{i}$. \\
According to \cite[Section 3]{BT}, the bi-double cover $Z$ is a
sub-variety of $B_{1}\times B_{2}$ which is Lagrangian. 
\end{proof}

Let us now discuss what is happens if we do not make assumption on
the transversality of the intersection of $A_{1}$ and $A_{1}'$.
Let us denote by $\mathbb{A}_{m}$ a surface singularity with germ
\[
\{x^{m+1}=y^{2}+z^{2}\}
\]
and by $\qa_{m}$ a curve singularity with germ $\{x^{m+1}=y^{2}\}$.
\\
Since $A_{1},A_{1}'$ are smooth, the singularities of $A_{1}+A_{1}'$
are of type $\qa_{2m-1}$, $m>0$. Let $s$ be a $\qa_{2m-1}$-singularity
of $A_{1}+A_{1}'$. Recall that $\tilde{B}_{1}$ is the cover of $X$
branched over $\sum_{i=1}^{16}A_{i}$. The curve singularity above
$s$ in $\pi_{1}^{*}A_{1}'\subset\tilde{B}_{1}$ is a $\frak{a}_{m-1}$
singularity (see e.g. \cite[Table 1, P. 109]{BPVdV}). \\
Thus above the singularity $s$ of type $\qa_{2m-1}$ of $A_{1}+A_{1}'$,
the surface $V$ has a singularity of type $\mathbb{A}_{m-1}$, (where
in fact a $\mathbb{A}_{0}$ (resp. $\frak{a}_{0}$) point is a smooth
point). \\
The singularities $\mathbb{A}_{m}$ are $ADE$ singularities and by
the Theorem of Brieskorn on simultaneous resolution of singularities,
they do not change the values of $K^{2}$ , $\chi$ and $p_{g}$ of
the surface $\tilde{V}$ which is the minimal resolution of $V$ (we
consider the two successive double covers $V\to\tilde{B}_{1}$ and
$\tilde{B}_{1}\to X$). \\
Thus the surface $Z$ obtained by taking the minimal desingularisation
of $V$ and the contraction of the $30$ exceptional curves has the
same invariants $\chi(Z)$, $K_{Z}^{2}$ and $p_{g}(Z)$ as if the
intersection of $A_{1}$ and $A_{1}'$ was transverse. We observe
that the image of the natural map $Z\to B_{1}\times B_{2}$ is also
a Lagrangian surface by \cite[Section 3]{BT}.

Let $\a_{m}$ be the number of $\frak{a}_{2m-1}$ singularities on
$A_{1}+A_{1}'$. Using Miyaoka's bound on the number of quotient singularities
on a surface of general type (here to be the surface $B_{3},$ the
double cover of $X$ branched over $A_{1}+A_{1}'$), one gets: 
\[
\sum(n-\frac{1}{n})\a_{n}\leq\frac{4}{3}k.
\]
For $k=1$, a configuration of $6\frak{a}_{1}$ singularities on $A_{1}+A_{1}'$
is the only possibility. For $k=2$, the possibilities are 
\[
10\frak{a}_{1},\,8\frak{a}_{1}+\frak{a}_{3},\,7\frak{a}_{1}+\frak{a}_{5},
\]
but we know from explicit computations in \cite{RRS} that for a generic
Abelian surface polarized by $M$ with $M^{2}=6$, the singularities
of $A_{1}+A_{1}'$ are $10\frak{a}_{1}$. For $k=3$ the possibilities
are 
\[
14\frak{a}_{1},\,12\frak{a}_{1}+\frak{a}_{3},\,10\frak{a}_{1}+2\frak{a}_{3},\,11\frak{a}_{1}+\frak{a}_{5},\,10\frak{a}_{1}+\frak{a}_{7}.
\]

\subsection{The $H$-constant of the curve $\protect\G$}

Let $X$ be a surface, $\mathcal{P}$ be a non-empty finite set points
on $X$ and let $\bar{X}\to X$ be the blow-up of $X$ at $\mathcal{P}$.
For a curve $C$ let $\bar{C}_{\mathcal{P}}$ be the strict transform
of $C$ on $\bar{X}$. The $H$-constant of $C$ is defined by 
\[
H(C)=\inf_{\mathcal{P}}\frac{(\bar{C}_{\mathcal{P}})^{2}}{\#\mathcal{P}}
\]
and the $H$-constant of $X$ is $H(X)=\inf_{C}H(C)$, where the infimum
is taken over reduced curves. The $H$-constants have been introduced
for studying the bounded negativity Conjecture, which predicts that
there exists a bound $b_{X}$ such that for any reduced curve $C$
on $X$, one has $C^{2}\geq b_{X}$.

Let $A$ be the generic Abelian surface polarized by $M$ with $M^{2}=k(k+1)$
and let $\Gamma$ be the curve with a unique singularity which is
of multiplicity $4k+2$ and is in the numerical equivalence class
of $4M$. One computes immediately 
\[
H(\G)=\G^{2}-(4k+2)^{2}=-4.
\]
For the moment, one do not know curves on Abelian surfaces which have
$H$-constants lower than $-4$. We use these curves in a more thorough
study of curves with low $H$-constants in \cite{RLow}.

\vspace{5mm}
\noindent Xavier Roulleau,
\\Aix-Marseille Universit\'e, CNRS, Centrale Marseille,
\\I2M UMR 7373,  
\\13453 Marseille, France
\\ \email{Xavier.Roulleau@univ-amu.fr}
%\\ {\tt Xavier.Roulleau@univ-amu.fr}
\urladdr{http://www.i2m.univ-amu.fr/perso/xavier.roulleau}
\vspace{0.3cm} \\
%\\ 
%
\noindent 
\author{Alessandra Sarti} \\
\address{Laboratoire de Math\'ematiques et Applications, UMR CNRS 7348, \\	
		Universit\'e de Poitiers, T\'el\'eport 2,\\ Boulevard Marie et Pierre Curie, 		\\
	86962 FUTUROSCOPE CHASSENEUIL, France}\\ 
\email{sarti@math.univ-poitiers.fr} \\
\vspace{2mm}
\urladdr{http://www-math.sp2mi.univ-poitiers.fr/~sarti/}

\begin{thebibliography}{99}

\bibitem{NikiPezzo} Alexeev V., Nikulin V., Del Pezzo and K3 surfaces. MSJ Memoirs, 15. Mathematical Society of Japan, Tokyo, 2006. xvi+149 pp. ISBN: 4-931469-34-5

\bibitem{AST} Artebani M, Sarti A., Taki S., K3 surfaces with non-symplectic automorphisms of prime order, Math. Z. (2011) 268 507--533

\bibitem{BN} Barth W., Nieto I., Abelian surfaces of type $(1,3)$ and quartic surfaces with $16$ skew lines, J. Algebraic Geom. 3 (1994), no. 2, 173--222

\bibitem{BPVdV} Barth W., Hulek K., Peters A.M., Van de Ven A., Compact complex surfaces, Second edition, Springer-Verlag, Berlin, 2004. xii+436 pp.

%\bibitem{BCS} Boissière S., Camere C., Sarti A., Classification of automorphisms on a deformation family of hyperkähler fourfolds by $p$-elementary lattices, Kyoto J. Math. 56 (2016), no. 3, 465--499.

\bibitem{BT} Bogomolov F, Tschinkel Y., Lagrangian subvarieties of Abelian fourfolds, Asian J. Math. 4 (2000), no. 1, 19--36. 

\bibitem{BOPY} Bryan J., Oberdieck G., R. Pandharipande R., Yin Q., Curve counting on abelian surfaces and threefolds, to appear in Algebr. Geom.

\bibitem{CMR} Ciliberto C., Mendes-Lopes M., Roulleau X., On Schoen surfaces, Comment. Math. Helv. 90 (2015), no. 1, 59--74


%\bibitem{GS1} Garbagnati A.,  Sarti A., Projective models of K3 surfaces with an even set, Adv. in Geom. 8 (2008), 413--440 

\bibitem{GS2} Garbagnati A., Sarti A., On symplectic and non-symplectic automorphisms of K3 surfaces, Rev. Mat. Iberoam. 29 (2013), no. 1, 135--162

\bibitem{GS} Garbagnati A., Sarti A., Kummer surfaces and K3 surfaces with $(\ZZ/2\ZZ)^4$ symplectic action, Rocky Mountain J., 46 (2016), no. 4, 1141--1205


\bibitem{GH} Gritsenko V., Hulek K., Minimal Siegel modular threefolds, Math. Proc. Cambridge Philos. Soc. 123 (1998), 461--485

\bibitem{Huyb} Huybrechts D., Lectures on K3 surfaces,  Cambridge Studies in Advanced Mathematics, 158, 2016. xi+485 pp

\bibitem{HLOY} Hosono S., Lian B.H., Oguiso K., Yau S.T., Kummer structures on a K3 surface - an old question of T. Shioda,  Duke Math. J. 120 (2003), no. 3, 635--647

\bibitem{Keum} Keum J.H., Automorphisms of Jacobian Kummer surfaces, Compositio Mathematica 107: 269--288, 1997.

\bibitem{Kondo} Kondo S., The automorphism group of a generic Jacobian Kummer surface, J. Alg. Geom. 7 (1998) 589--609.

\bibitem{Lange} Lange H., Principal polarizations on products of elliptic curves, Contemp. Math., 397, Amer. Math. Soc., Providence, RI, 2006.

\bibitem{Mc} McMullen C., K3 surfaces, entropy and glue, J. Reine Angew. Math. 658 (2011), 1--25

\bibitem{Morrison} Morrison D., On K3 surfaces with large Picard number, Invent. Math. 75 (1984), no. 1, 105--121.

%\bibitem{Namikawa} Namikawa Y., Periods of Enriques surfaces, Math. Ann. 270 (1985), no. 2, 201--222.

\bibitem{NarNori}   Narasimhan M.S., Nori M.V., Polarisations on an abelian variety, Proc. Ind. Acad. Sci. (Math), Volume 90, Number 2, April 1981, 125--128

\bibitem{Nikulin} Nikulin V., Kummer surfaces, Izv. Akad. Nauk SSSR Ser. Mat. 39 (1975), 278--293. English translation: Math. USSR. Izv, 9 (1975), 261--275.

\bibitem{Orlov} Orlov D.O., On equivalences of derived categories of coherent sheaves on abelian varieties, Izv. Ross. Akad. Nauk Ser. Mat. 66 (2002), no. 3, 131--158; translation in  Izv. Math. 66 (2002), no. 3, 569--594 

\bibitem{Polizzi} Polizzi F., Monodromy representations and surfaces with maximal Albanese dimension, Bollettino dell'Unione Matematica Italiana, 1-13, 2017

\bibitem{Pene}  Penegini M., The classification of isotrivially fibred surfaces with $p_g = q = 2$. Collect. Math., 62(3):239--274, 2011. With an appendix by S\"onke Rollenske 

\bibitem{reid}  Reid M., Chapters on algebraic surfaces, Complex algebraic geometry (Park City, UT, 1993), IAS/Park City Math. Ser., vol. 3, Amer. Math. Soc., Providence, RI, 1997, pp. 3--159. MR 1442522 (98d:14049)

\bibitem{RRS} Rito C., Roulleau X.,  Sarti A., On explicit Schoen surfaces, to appear in Algebr. Geom.

\bibitem{RoulleauIMRN} Roulleau X., Bounded negativity, Miyaoka-Sakai inequality and elliptic curve configurations, Int Math Res Notices (2017) 2017 (8): 2480--2496

\bibitem{RLow} Roulleau X., Curves with low Harbourne constants on Kummer and Abelian surfaces, to appear in Rend. Circ. Mat. Palermo, II. Ser.


%\bibitem{RU} Roulleau X.,  Urz\'ua G., Chern slopes of simply connected complex surfaces of general type  are dense in $[2,3]$, Annals of  Math. 182 (2015), 287--306

\bibitem{SD} Saint-Donat B., Projective models of K3 surfaces. Amer. J. Math. 96 (1974), 602--639

\bibitem{SvG} Sarti A., van Geemen B., Nikulin involutions on K3 surfaces, Math. Z. 255 (2007), no. 4, 731--753.

\bibitem{Sh1}  Shioda T., Some remarks on abelian varieties, J. Fac. Sci. Univ. Tokyo, Sect. IA, 24 (1977) 11--21.

%\bibitem{Walker}  Walker R.J., Algebraic Curves, Princeton Mathematical Series, vol. 13. Princeton Univ. Press, 1950. x+201 pp
%\bibitem{Shimada} Shimada I., An algorithm to compute automorphism groups of K3 surfaces and an application to singular K3 surfaces, Int. Math. Res. Not. IMRN 2015, no. 22, 11961--12014





%http://www-math.mit.edu/~poonen/papers/curvetorsion.pdf

%\bibitem{BandCo} Bauer T., Di Rocco S., Harbourne B., Huizenga J., Lundman A., Pokora P., Szemberg T., Bounded negativity and arrangement of lines,  Int. Math. Res. Notices.  (2015) 2015 (19): 9456--9471

%\bibitem{BandCoBis} Bauer T., Harbourne B., Knutsen A.L.; Küronya A., Müller-Stach S., Roulleau X., Szemberg T., Negative curves on algebraic surfaces, Duke Math. J. 162 (2013), 1877--1894.


%\bibitem{PT} Pokora P., Tutaj-Gasi\'nska H., Harbourne constants and conic configurations on the projective plane, Math. Nachr. 289 (2016), no. 7, 888--894.

%\bibitem{P} Pokora P., Harbourne constants and arrangements of lines on smooth hypersurfaces in $\mathbb{P}^4(\mathbb{C})$, Taiwanese J. Math. 20 (2016), no. 1, 25--31


\end{thebibliography}
\end{document}